\documentclass[12pt,leqno]{article}
\usepackage{amssymb}
\usepackage[mathscr]{eucal}
\usepackage{amsmath,amssymb,latexsym,theorem,enumerate,ifthen}
\usepackage[english]{babel}
\usepackage{color,url}
\usepackage{appendix}
\usepackage{graphicx}
\usepackage{subcaption}

\newcommand{\CC}{\mathbb{C}}

\newcommand{\NN}{\mathbb{N}}
\newcommand{\QQ}{\mathbb{Q}}
\newcommand{\RR}{\mathbb{R}}

\newcommand{\ZZ}{\mathbb{Z}}

\newcommand{\bx}{{\boldsymbol{x}}}

\newcommand{\blambda}{{\boldsymbol{\lambda}}}

\newcommand{\bone}{{\boldsymbol{1}}}

\newcommand{\cB}{{\mathcal B}}

\newcommand{\cF}{{\mathcal F}}

\newcommand{\cV}{{\mathcal V}}

\newcommand{\cX}{{\mathcal X}}

\DeclareMathOperator*{\argmax}{arg\,max}

\newcommand{\PP}{\operatorname{\mathbb{P}}}

\newcommand{\sign}{\operatorname{sign}}

\newcommand{\vare}{\varepsilon}
\newcommand{\comment}[1]{}

\renewcommand{\mid}{\,|\,}

\renewcommand{\leq}{\leqslant}
\renewcommand{\geq}{\geqslant}

\newcommand{\proofend}{\hfill\mbox{$\Box$}}

\numberwithin{equation}{section}

\theoremstyle{change} \theorembodyfont{\em}
\newtheorem{Lem}{Lemma.}[section]
\newtheorem{Thm}[Lem]{Theorem.}

\newtheorem{Cor}[Lem]{Corollary.}

\theorembodyfont{\rm}
\newtheorem{Def}[Lem]{Definition.}
\newtheorem{Rem}[Lem]{Remark.}

\newtheorem{Ex}[Lem]{Example.}

\def\OnlyOnArXiv#1#2{\ifthenelse{\equal{#1}{Y}}{#2}{}}

\long\def\Eq#1#2{\ifthenelse{\equal{#1}{*}}
  {\begin{equation*}\begin{aligned}#2\end{aligned}\end{equation*}}
  {\begin{equation}\begin{aligned}\label{#1}#2\end{aligned}\end{equation}}}

\newenvironment{proof}{\noindent{\bf Proof.}}{\proofend}

\begin{document}

\begin{center}
 {\bfseries\Large Determining classes for generalized $\psi$-estimators}

\vspace*{3mm}

{\sc\large
  M\'aty\'as $\text{Barczy}^{*,\diamond}$,
  Zsolt $\text{P\'ales}^{**}$ }

\end{center}

\vskip0.2cm

\noindent
 * HUN-REN–SZTE Analysis and Applications Research Group,
   Bolyai Institute, University of Szeged,
   Aradi v\'ertan\'uk tere 1, H--6720 Szeged, Hungary.

\noindent
 ** Institute of Mathematics, University of Debrecen,
    Pf.~400, H--4002 Debrecen, Hungary.

\noindent e-mail: barczy@math.u-szeged.hu (M. Barczy),
                  pales@science.unideb.hu  (Zs. P\'ales).

\noindent $\diamond$ Corresponding author.

\vskip0.2cm


{\renewcommand{\thefootnote}{}
\footnote{\textit{2020 Mathematics Subject Classifications\/}:
 {Primary: 62F10, Secondary: 26E60, 26D15.} }
\footnote{\textit{Key words and phrases\/}:
 generalized $\psi$-estimator, $Z$-estimator, determining class, comparative function, Schweitzer's inequality, Kantorovich's inequality.}
\vspace*{0.2cm}
\footnote{M\'aty\'as Barczy is supported by the Ministry of Innovation and Technology of Hungary from the National Research, Development and Innovation Fund, project no.\ TKP2021-NVA-09.
Zsolt P\'ales is supported by the K-134191 NKFIH Grant.}}

\vspace*{-10mm}

\begin{abstract}
We prove that the values of a generalized $\psi$-estimator (introduced by Barczy and P\'ales in 2025) on samples of arbitrary length but having only two different observations uniquely determine the values of the estimator on any sample of arbitrary length without any restriction on the number of different observations.
In other words, samples of arbitrary length but having only two different observations form a determining class for generalized $\psi$-estimators.
We also obtain a similar statement for the comparison of generalized $\psi$-estimators using comparative functions,
 and, as a corollary of this result, we derive the Schweitzer's inequality (also called Kantorovich's inequality).
\end{abstract}


\section{Introduction}
\label{section_intro}

Determining classes play important roles in many branches of mathematics, for example,
in real and complex analysis and in probability theory as well.
In this paper, we will provide an explicitly given (non-trivial) determining class (see Definition \ref{Def_det_class})
 for some generalized $\psi$-estimators, introduced by Barczy and P\'ales \cite{BarPal2}.
First, as a motivation, we will  recall some notable results for determining classes in the above mentioned fields.

Given a measurable space $(\Omega,\cF)$, a subset $\cV$ of $\cF$ is called a determining class for $(\Omega,\cF)$ if, for any pair of probability measures $\PP_1$ and $\PP_2$ on $(\Omega,\cF)$, whenever the equality $\PP_1(A) = \PP_2(A)$ holds for all $A\in\cV$, then it also holds for all $A\in\cF$.
For example, if $\Omega$ is a finite set and $\cF$ is the power set $2^\Omega$ of $\Omega$ (i.e., the collection of the subsets of $\Omega$), then the set $\cV$ consisting of all the singletons $\{\omega\}$, $\omega\in\Omega$, is a determining class for $(\Omega, 2^\Omega)$.
Turning back to the general setup, note that if $\cV$ is a subset of $\cF$ such that it generates $\cF$ (i.e., the generated $\sigma$-algebra by $\cV$ equals $\cF$), then $\cV$ is not necessarily a determining class for $(\Omega,\cF)$.
Further, it is known that if $S$ is a metric space and  $\cF$ is the Borel $\sigma$-algebra $\cB(S)$ corresponding to the metric of $S$, then the class of closed subsets of $S$ forms a determining class for $(S,\cB(S))$, see, e.g., Billingsley \cite[Theorem 1.1]{Bil}.

In real analysis, a Hamel basis $H$ is any basis of $\RR$
viewed as a vector space over the field of the rational  numbers.
It is known that for every function $g:H\to\RR$, there exists a unique additive function $f:\RR\to\RR$ such that $f=g$ on $H$, see, e.g., Kuczma \cite[Theorem 5.2.2]{Kuc2009}.
In other words, we could say that any Hamel basis of $\RR$ is a determining class
for the set of real-valued additive functions defined on $\RR$.

In complex analysis, the so-called Identity Theorem for analytic functions states that
given two analytic functions $f,g:D\to\CC$, defined on a domain $D$ (a nonempty, open and connected subset of the set complex numbers $\CC$), if $f=g$ on some subset $S$ of $D$, where $S$ has an accumulation point in $D$,
 then $f = g$ on $D$, see, e.g., Freitag and Busam \cite[Theorem III.3.2]{FreBus}.
As a consequence, an analytic function is completely determined
 by its values on a countable subset which contains a converging sequence together with its limit.
In other words, we could say that such countable subsets form a determining class for the analytic functions defined on a domain.
Cauchy's integral formulas could be also interpreted in terms of determining classes,
 since the essence of these formulas is that they compute the values of an analytic function
 in the interior of a non-degenerate open disk from its values on the boundary, see, e.g., Freitag and Busam \cite[Theorem II.3.2]{FreBus}.

Throughout this paper, we fix the following notations.
Let $\NN$, $\ZZ_+$, $\QQ$, $\RR$, $\RR_+$, and $\RR_{++}$ denote the sets of positive integers, non-negative integers, rational numbers, real numbers, non-negative real numbers, and positive real numbers, respectively.
A real interval will be called nondegenerate if it contains at least two distinct points.
Given a nondegenerate interval $\Theta\subseteq \RR$ and a function $f:\Theta\to \RR$, let us define
 \[
   \argmax_{t\in\Theta} f(t) := \big\{ t\in\Theta : f(s)\leq f(t)\;\; \text{for all $s\in\Theta$} \big\}.
 \]
Given a function $f:\RR^2\to\RR$, the partial derivative of $f$ with respect to its first and second variables are denoted by
 $\partial_1 f$ and $\partial_2 f$, respectively.
For each $n\in\NN$, let us also introduce the set $\Lambda_n:=\RR_+^n\setminus\{(0,\ldots,0)\}$.

Next, we motivate and recall the notions of $\psi$-estimators and generalized $\psi$-estimators that are the main objects of the present paper.
In statistics, $M$-estimators play a fundamental role, and a special subclass, the class of $\psi$-estimators (also called $Z$-estimators), is also in the heart of investigations.
Let $(X,\cX)$ be a measurable space, $\Theta$ be a Borel subset of $\RR$, and $\psi:X\times\Theta\to\RR$ be a function such that for all $t\in\Theta$, the function $X\ni x\mapsto \psi(x,t)$ is measurable with respect to the sigma-algebra $\cX$.
Let $(\xi_n)_{n\in\NN}$ be a sequence of independent and identically distributed (i.i.d.) random variables
 with values in $X$ such that the distribution of $\xi_1$ depends on an unknown parameter $\vartheta \in\Theta$.
For each $n\in\NN$, Huber \cite{Hub64, Hub67} among others introduced an important estimator of $\vartheta$ based on the observations $\xi_1,\ldots,\xi_n$ as a solution $\widehat\vartheta_{n,\psi}(\xi_1,\ldots,\xi_n)$ of the
 equation (with respect to the unknown parameter):
 \[
  \sum_{i=1}^n \psi(\xi_i,t)=0, \qquad t\in\Theta.
 \]
In the statistical literature, one calls $\widehat\vartheta_{n,\psi}(\xi_1,\ldots,\xi_n)$ a $\psi$-estimator of the unknown parameter $\vartheta\in\Theta$ based on the i.i.d.\ observations $\xi_1,\ldots,\xi_n$, while other authors call it a $Z$-estimator (the letter Z refers to ''zero'').
In fact, $\psi$-estimators are special $M$-estimators (where the letter $M$ refers to ''maximum likelihood-type'')
 that were also introduced by Huber \cite{Hub64, Hub67}.
For a detailed exposition of $M$-estimators and $\psi$-estimators, see, e.g., Kosorok \cite[Sections 2.2.5 and 13]{Kos} or van der Vaart \cite[Section 5]{Vaa}.
In our recent paper Barczy and P\'ales \cite{BarPal2}, we introduced the notion of weighted generalized $\psi$-estimators
 (recalled also in parts (ii) and (iv) of Definition \ref{Def_Tn}), and we studied their existence and uniqueness.

Throughout this paper, let $X$ be a nonempty set, $\Theta$ be a nondegenerate open interval of $\RR$.
Let $\Psi(X,\Theta)$ denote the class of real-valued functions $\psi:X\times\Theta\to\RR$ such that,
 for all $x\in X$, there exist $t_+,t_-\in\Theta$ such that $t_+<t_-$ and $\psi(x,t_+)>0>\psi(x,t_-)$.
Roughly speaking, a function $\psi\in\Psi(X,\Theta)$ satisfies the following property: for all $x\in X$,
 the function $t\ni\Theta\mapsto \psi(x,t)$ changes sign (from positive to negative)
 on the interval $\Theta$ at least once.

In practical applications, $X$ could be a subset of $\RR^d$ (corresponding to multidimensional data), or that of $L^2([0,1])$,
 the space of square-integrable real-valued measurable functions defined on $[0,1]$ (corresponding to functional data).

\begin{Def}\label{Def_sign_change}
Let $\Theta$ be a nondegenerate open interval of $\RR$. For a function $f:\Theta\to\RR$, consider the following three level sets
\[
  \Theta_{f>0}:=\{t\in \Theta: f(t)>0\},\qquad
  \Theta_{f=0}:=\{t\in \Theta: f(t)=0\},\qquad
  \Theta_{f<0}:=\{t\in \Theta: f(t)<0\}.
\]
We say that $\vartheta\in\Theta$ is a \emph{point of sign change (of decreasing type) for $f$} if
 \[
 f(t) > 0 \quad \text{for $t<\vartheta$,}
   \qquad \text{and} \qquad
    f(t)< 0 \quad  \text{for $t>\vartheta$.}
 \]
\end{Def}

Note that there can exist at most one element $\vartheta\in\Theta$ which is a point of sign change for $f$.
Further, if $f$ is continuous at a point $\vartheta$ of sign change, then $\vartheta$ is the unique zero of $f$.

\begin{Def}\label{Def_Tn}
We say that a function $\psi\in\Psi(X,\Theta)$
  \begin{enumerate}[(i)]
    \item \emph{possesses the property $[C]$ (briefly, $\psi$ is a $C$-function)} if
           it is continuous in its second variable, i.e., if, for all $x\in X$,
           the mapping $\Theta\ni t\mapsto \psi(x,t)$ is continuous.
    \item \emph{possesses the property $[T_n]$ (briefly, $\psi$ is a $T_n$-function)
           for some $n\in\NN$} if there exists a mapping $\vartheta_{n,\psi}:X^n\to\Theta$ such that,
           for all $\pmb{x}=(x_1,\dots,x_n)\in X^n$ and $t\in\Theta$,
           \begin{align*}
             \psi_{\pmb{x}}(t):=\sum_{i=1}^n \psi(x_i,t) \begin{cases}
                 > 0 & \text{if $t<\vartheta_{n,\psi}(\pmb{x})$,}\\
                 < 0 & \text{if $t>\vartheta_{n,\psi}(\pmb{x})$},
            \end{cases}
           \end{align*}
          that is, for all $\pmb{x}\in X^n$, the value $\vartheta_{n,\psi}(\pmb{x})$ is a point of sign change for the function $\psi_{\pmb{x}}$. If there is no confusion, instead of $\vartheta_{n,\psi}$ we simply write $\vartheta_n$.
          We may call $\vartheta_{n,\psi}(\pmb{x})$ as a generalized $\psi$-estimator for
         some unknown parameter in $\Theta$ based on the realization $\bx=(x_1,\ldots,x_n)\in X^n$. If, for each $n\in\NN$, $\psi$ is a $T_n$-function, then we say that \emph{$\psi$ possesses the property $[T]$ (briefly, $\psi$ is a $T$-function)}.
    \item \emph{possesses the property $[Z_n]$ (briefly, $\psi$ is a $Z_n$-function) for some $n\in\NN$} if it is a $T_n$-function and
    \[
   \psi_{\pmb{x}}(\vartheta_{n,\psi}(\pmb{x}))=\sum_{i=1}^n \psi(x_i,\vartheta_{n,\psi}(\pmb{x}))= 0
    \qquad \text{for all}\quad \pmb{x}=(x_1,\ldots,x_n)\in X^n.
    \]
    If, for each $n\in\NN$, $\psi$ is a $Z_n$-function, then we say that \emph{$\psi$ possesses the property $[Z]$ (briefly, $\psi$ is a $Z$-function)}.
    \item \emph{possesses the property $[T_n^{\pmb{\lambda}}]$ for some $n\in\NN$ and $\pmb{\lambda}=(\lambda_1,\ldots,\lambda_n)\in\Lambda_n$ (briefly, $\psi$ is a $T_n^{\pmb{\lambda}}$-function)} if there exists a mapping $\vartheta_{n,\psi}^{\pmb{\lambda}}:X^n\to\Theta$ such that, for all $\pmb{x}=(x_1,\dots,x_n)\in X^n$ and $t\in\Theta$,
          \begin{align*}
           \psi_{\pmb{x},\pmb{\lambda}}(t):= \sum_{i=1}^n \lambda_i\psi(x_i,t) \begin{cases}
                 > 0 & \text{if $t<\vartheta_{n,\psi}^{\pmb{\lambda}}(\pmb{x})$,}\\
                 < 0 & \text{if $t>\vartheta_{n,\psi}^{\pmb{\lambda}}(\pmb{x})$},
             \end{cases}
           \end{align*}
           that is, for all $\pmb{x}\in X^n$, the value $\vartheta_{n,\psi}^{\pmb{\lambda}}(\pmb{x})$ is
           a point of sign change for the function $\psi_{\pmb{x},\pmb{\lambda}}$.
           If there is no confusion, instead of $\vartheta_{n,\psi}^{\pmb{\lambda}}$ we simply write $\vartheta_n^{\pmb{\lambda}}$.
          We may call $\vartheta_{n,\psi}^{\pmb{\lambda}}(\pmb{x})$
          as a weighted generalized $\psi$-estimator for some unknown parameter in $\Theta$ based
          on the realization $\bx=(x_1,\ldots,x_n)\in X^n$ and weights $(\lambda_1,\ldots,\lambda_n)\in\Lambda_n$.
    \item \emph{possesses the property $[Z_n^{\pmb{\lambda}}]$ for some $n\in\NN$ and $\pmb{\lambda}=(\lambda_1,\ldots,\lambda_n)\in\Lambda_n$ (briefly, $\psi$ is a $Z_n^{\pmb{\lambda}}$-function)} if it is a $T_n^{\pmb{\lambda}}$-function and
    \[
        \psi_{\pmb{x,\blambda}}(\vartheta_{n,\psi}^\blambda(\pmb{x}))
           =\sum_{i=1}^n \lambda_i \psi(x_i,\vartheta_{n,\psi}^\blambda(\bx))= 0
    \qquad \text{for all}\quad \pmb{x}=(x_1,\ldots,x_n)\in X^n.
    \]
    \item \emph{possesses the property $[W_n]$ for some $n\in\NN$ (briefly, $\psi$ is a $W_n$-function)}
          if it is a $T_n^{\pmb{\lambda}}$-function for all $\pmb{\lambda}\in\Lambda_n$. If, for each $n\in\NN$, $\psi$ is a $W_n$-function, then we say that \emph{$\psi$ possesses the property $[W]$ (briefly, $\psi$ is a $W$-function)}.

   \end{enumerate}
\end{Def}

It can be seen that if $\psi$ is continuous in its second variable, and, for some $n\in\NN$, it is a $T_n$-function, then it also a $Z_n$-function.
Further, if $\psi\in\Psi(X,\Theta)$ is a $T_n$-function for some $n\in\NN$, then $\vartheta_{n,\psi}$ is symmetric in the sense that
$\vartheta_{n,\psi}(x_1,\ldots,x_n) = \vartheta_{n,\psi}(x_{\pi(1)},\ldots,x_{\pi(n)})$ holds for all $x_1,\ldots,x_n\in X$ and all permutations $(\pi(1),\ldots,\pi(n))$ of $(1,\ldots,n)$.

Assume that $\psi$ possesses property $[W]$.
Then, by definition, it also possesses the property $[T]$, and, for all $n\in\NN$, $\vartheta_{n,\psi}=\vartheta_{n,\psi}^{(1,\dots,1)}$ holds on $X^n$.
More generally, for all $k,n_1,\dots,n_k\in\NN$ and $x_1,\dots,x_k\in X$, we have
\Eq{*}{
  \vartheta_{n_1+\dots+n_k,\psi}(\underbrace{x_1,\dots,x_1}_{n_1},\dots,\underbrace{x_k,\dots,x_k}_{n_k})
  =\vartheta_{k,\psi}^{(n_1,\dots,n_k)}(x_1,\dots,x_k).
}

Given $q\in\NN$ and properties $[P_1], \ldots, [P_q]$ which can be any of the properties introduced in Definition~\ref{Def_Tn}, the subclass of $\Psi(X,\Theta)$ consisting of elements possessing the properties $[P_1],\ldots,[P_q]$, will be denoted by $\Psi[P_1,\ldots,P_q](X,\Theta)$.

In our paper Barczy and P\'ales \cite[Section 4]{BarPal2}, we presented several examples for statistical estimators,
 which are special generalized $\psi$-estimators.
The first group of examples includes several well-known descriptive statistics that
 can be considered as special generalized $\psi$-estimators, namely, the empirical median, the empirical quantiles
 and the empirical expectiles.
The second group of examples contains the class of $\psi$-estimators recently used by Mathieu \cite{Mat},
 and some $\psi$-estimators that are important in robust statistics.
The third group of examples includes maximum likelihood estimators of some parameter of absolutely continuous distributions
 (for example, we considered a mixture Gaussian distribution).
Further, in Barczy and P\'ales \cite[Definition 4, Proposition 6 and the paragraph after them]{BarPal2},
 we also introduced Bajraktarevi\'{c}-type (in particular, quasi-arithmetic-type) $\psi$-estimators as special generalized $\psi$-estimators.
These special and well-known cases of generalized $\psi$-estimators may indicate that it is worth studying general properties of this class of estimators.
In the present paper, we consider the problem of finding an explicitly given determining class for generalized $\psi$-estimators.

Next, we recall Theorem 2.1 and Corollary 3.8 from our paper Barczy and P\'ales \cite{BarPal4},
 which are going to play an essential role in the proof of our main results, Theorems \ref{Thm_2_to_n} and \ref{Thm_2_to_n+w}.
 For their formulations, we will adopt the following convention: if $k,n_1,\dots,n_k\in\NN$ and
 $\bx_1=(x_{1,1},\dots,x_{1,n_1})\in X^{n_1}$, \dots, $\bx_k=(x_{k,1},\dots,x_{k,n_k})\in X^{n_k}$,
 then their \emph{concatenation} $(\bx_1,\dots,\bx_k)\in X^{n_1}\times\cdots\times X^{n_k}$ is identified by
\[
 (x_{1,1},\dots,x_{1,n_1}, \dots,
 x_{k,1},\dots,x_{k,n_k})\in X^{n_1+\dots+n_k}.
\]
Furthermore, for each $n\in\NN$ and $x\in X$, let $n\odot x$ denote the $n$-tuple $(x,\ldots,x)\in X^n$.

\begin{Thm}\label{Thm_psi_becsles_mean_prop}{\rm (\cite[Theorem 2.1]{BarPal4})}
Let $\psi\in\Psi[T](X,\Theta)$.
Then, for each $k,n_1,\dots,n_k\in\NN$ and $\bx_1\in X^{n_1}$, \dots, $\bx_k\in X^{n_k}$, we have
\begin{align}\label{help_mean_prop_csop}
\begin{split}
 \min(\vartheta_{n_1}(\bx_1),\dots,\vartheta_{n_k}(\bx_k))
  \leq \vartheta_{n_1+\dots+n_k}(\bx_1,\dots,\bx_k)\leq \max(\vartheta_{n_1}(\bx_1),\dots,\vartheta_{n_k}(\bx_k)).
\end{split}
\end{align}
Furthermore, if $\psi\in\Psi[Z](X,\Theta)$ and not all of the values $\vartheta_{n_1}(\bx_1),\dots,\vartheta_{n_k}(\bx_k)$ are equal, then both inequalities in \eqref{help_mean_prop_csop} are strict.
\end{Thm}

The property \eqref{help_mean_prop_csop} is a kind of mean-type property for generalized $\psi$-estimators.

\begin{Lem}\label{Lem_infinit}{\rm (\cite[Corollary 3.8]{BarPal4})}
Let $\psi\in\Psi[T,W_2](X,\Theta)$. Then the generalized $\psi$-estimator
$\vartheta_\psi:\bigcup_{n=1}^\infty X^n\to\Theta$ defined by
 \[
    \vartheta_\psi(\bx):= \vartheta_{n,\psi}(\bx),
     \qquad n\in\NN,\quad \bx\in X^n,
 \]
possesses the property of sensitivity, that is, for all $x,y\in X$, $u,v\in\Theta$
 with $\vartheta_{1,\psi}(x)<u<v<\vartheta_{1,\psi}(y)$, there exist $k,m\in\NN$ such that
 \[
  u < \vartheta_{k+m,\psi}(k\odot x, m\odot y) < v.
 \]
\end{Lem}

The paper is organized as follows.
In Section \ref{section_formula}, we derive an explicitly given determining class for generalized $\psi$-estimators.
Namely, given a function $\psi\in\Psi(X,\Theta)$ having the properties $[T]$ and $[W_2]$, we prove that the values of the corresponding generalized $\psi$-estimator on samples of arbitrary length but having only two different observations uniquely determine the values of the estimator on any sample of arbitrary length without any restriction on the number of different observations (see Theorems \ref{Thm_2_to_n} and \ref{Thm_2_to_n+w}) with a function $\psi\in\Psi(X,\Theta)$ having the properties $[T]$ and $[W_2]$.
In other words, samples of arbitrary length but having only two different observations form a determining class for generalized $\psi$-estimators.
We also provide three applications of Theorem \ref{Thm_2_to_n}, the first and second are related to the maximum likelihood estimator of some parameter of certain absolutely continuous random variables, and the third is about expectiles.

In Example \ref{Ex3}, we can approximate the unique solution of the equation
\Eq{p-odd}{\sum_{i=1}^n(x_i-t)^p = 0,\qquad t\in\RR,}
where $n\in\NN$, $x_1,\ldots,x_n\in\RR$, and $p\geq 1$ is an odd integer, using the explicit solutions of the equations having the form $\lambda_1(x_1-t)^p + \lambda_2(x_2-t)^p=0$, $t\in\RR$, where $x_1,x_2\in\RR$ and $(\lambda_1,\lambda_2)\in\Lambda_2$.
Note that the equation \eqref{p-odd} is nothing else but the likelihood equation for $t$ when we estimate the parameter $t$ of a generalized normal distribution having location parameter $t$, scale parameter $(p+1)^{\frac{1}{p+1}}$ and shape parameter $p+1$ using the maximum likelihood method supposing that $p$ is known (for more details, see Example \ref{Ex3}).

In Example \ref{Ex6}, we consider a similar question for the equation
 \begin{align}\label{help_28}
   \frac{\alpha}{\alpha+1} = \frac{\mu}{n}\sum_{i=1}^n \frac{1}{x_i+\mu},\qquad \mu\in\RR_{++},
 \end{align}
 where $n\in\NN$, $x_1,\ldots,x_n\in\RR_{++}$, and $\alpha\in\RR_{++}$ are given.
Note that the equation \eqref{help_28} is nothing else but the likelihood equation for $\mu$ when we estimate the parameter $\mu$ of a Lomax (Pareto type II) distribution with parameters $\alpha>0$ and $\mu>0$ using the maximum likelihood method
supposing that $\alpha$ is known (for more details, see Example \ref{Ex6}).

In Example \ref{Ex_Expectile}, we demonstrate how one can use Theorem \ref{Thm_2_to_n} in order to approximate the empirical expectile of any sample by the empirical expectiles of subsamples consisting only of two observations, that are nothing else but some weighted means.

In Section \ref{section_comparative},
 we investigate the problem of comparison of generalized $\psi$-estimators
 using so-called comparative functions (continuous functions on $\Theta\times\Theta$
 that are increasing and decreasing in their first and second variable, respectively,
 see Definition \ref{Def_comparative_function}).
It turns out that two generalized $\psi$-estimators are comparable using a given comparative function
 for any sample of arbitrary length (without any restriction on the number of different observations)
 if and only if they are comparable for any sample of arbitrary length but having only two different observations
 (for a precise statement, see Theorem \ref{Thm_comparative_function}).
As an application of this result, in Corollary \ref{Cor_comparative_function}, we derive the Schweitzer’s inequality (also called Kantorovich’s inequality), see, e.g., formula (14) on page 202 in Bullen et al.\ \cite{Bul03}.

\section{Determining classes for generalized $\psi$-estimators}
\label{section_formula}

First, we recall and introduce some notations that are needed to formulate our forthcoming main results.
Recall that, for each $k\in\NN$ and $x\in X$, the notation $k\odot x$ denotes the $k$-tuple $(x,\ldots,x)\in X^k$, and let $0\odot x$ be the empty tuple (a neutral element for the operation concatenation introduced earlier).
For a finite set $F$, the cardinality of $F$ will be denoted by $\# F$.
Furthermore, for $k,n\in\NN$ with $k<n$ and for a matrix $\mathsf{m}=(\mathsf{m}_{i,j})_{i,j=k}^n \in\ZZ_+^{(n-k+1)\times (n-k+1)}$,
 the set $\Delta(\mathsf{m})\subseteq\{k,\dots,n\}^2$ is defined by
 \begin{align}\label{help_Delta_m}
 \Delta(\mathsf{m}):=\big\{(i,j)\colon k\leq i<j
 \leq n,\ \mathsf{m}_{i,j}+\mathsf{m}_{j,i}\geq 1\big\}.
 \end{align}

\begin{Def}\label{Def_det_class}
Let $q\in\NN$ and $[P_1], \ldots, [P_q]$ be some of the properties introduced in Definition~\ref{Def_Tn}.
A subset $S\subseteq\bigcup_{n=1}^\infty X^n$ is said to be a determining class for $\Psi[P_1,\ldots,P_q](X,\Theta)$ if, for any pair of functions $\psi,\varphi\in\Psi[P_1,\ldots,P_q](X,\Theta)$, whenever the equality
\[
 \vartheta_{n,\psi}(\bx)=\vartheta_{n,\varphi}(\bx)
\]
holds for all $n\in\NN$ and $\bx\in S\cap X^n$, then it also holds for all $n\in\NN$ and $\bx\in X^n$.
\end{Def}

As the main result of this paper, in the case of $q=2$, $[P_1]=[T]$ and $[P_2]=[W_2]$, we construct a nontrivial determining class for $\Psi[T,W_2](X,\Theta)$, and, in addition, we also provide an explicit formula for the generalized $\psi$-estimator in terms of its restriction to the determining class.
We note that, provided that $\psi\in\Psi[Z_1](X,\Theta)$, the function $\psi$ belongs to $\Psi[W_2](X,\Theta)$ if and only if it belongs to $\Psi[W](X,\Theta)$
 (i.e., $\psi\in\Psi[W_n](X,\Theta)$ for each $n\in\NN$), see Barczy and P\'ales \cite[Corollary 1]{BarPal2}.
Consequently, $\Psi[Z_1,W_2](X,\Theta)\subseteq\Psi[T,W_2](X,\Theta)$.

\begin{Thm}\label{Thm_2_to_n}
Let $\psi\in\Psi[T,W_2](X,\Theta)$.
Then, for all $n\geq 2$, $n\in\NN$, and $x_1,\ldots,x_n\in X$, we have
\begin{align}\label{help_psi_2sample_formula}
 \begin{split}
   \vartheta_{n}(x_1,\dots,x_n)
   &=\sup_{\mathsf{m}\in \mathsf{M}_n}\min_{(i,j)\in\Delta(\mathsf{m})}
    \vartheta_{\mathsf{m}_{i,j}+\mathsf{m}_{j,i}}(\mathsf{m}_{i,j}\odot x_i, \mathsf{m}_{j,i}\odot x_j)\\
   &=\inf_{\mathsf{m}\in \mathsf{M}_n}\max_{(i,j)\in\Delta(\mathsf{m})}
    \vartheta_{\mathsf{m}_{i,j}+\mathsf{m}_{j,i}}(\mathsf{m}_{i,j}\odot x_i, \mathsf{m}_{j,i}\odot x_j),
 \end{split}
 \end{align}
 where the set $\mathsf{M}_n$ of nonnegative integer-valued $n\times n$ matrices is defined by
 \begin{align*}
  \mathsf{M}_n:=\bigcup_{k=1}^\infty\bigg\{ \mathsf{m}=(\mathsf{m}_{i,j})_{i,j=1}^n \in\ZZ_+^{n\times n}\colon\!\! \sum_{j\in\{1,\dots,n\}\setminus\{i\}}\!\! \mathsf{m}_{i,j} =k,\; i\in\{1,\ldots,n\}\
  \hbox{ and }\ \#\Delta(\mathsf{m})\leq n-1 \bigg\}.
 \end{align*}
\end{Thm}

Note that, for each $n\geq 2$, $n\in\NN$, and $\mathsf{m}\in\mathsf{M}_n$, the matrix $\mathsf{m}$ can have at most $2n-2$ non-zero off-diagonal entries.
In Example \ref{Ex3}, we write the set $\mathsf{M}_3$ in a more simpler form.
To verify the first equality in \eqref{help_psi_2sample_formula}, we are going to prove the following more general theorem
 from which the first equality in \eqref{help_psi_2sample_formula} can be deduced with the choices $k_1:=\cdots:=k_n:=1$.
The proof of the second equality in \eqref{help_psi_2sample_formula} is completely analogous, therefore, it is omitted.

\begin{Thm}\label{Thm_2_to_n+w}
Let $\psi\in\Psi[T,W_2](X,\Theta)$.
Then, for all $n\geq 2$, $n\in\NN$, $x_1,\ldots,x_n\in X$, and $k_1,\dots,k_n\in\NN$, we have
\Eq{help_psi_2sample_formula_w}{
   \vartheta_{k_1+\dots+k_n}&(k_1\odot x_1,\dots,k_n\odot x_n)\\
   &=\sup_{\mathsf{m}\in \mathsf{M}_n(k_1,\dots,k_n)}\min_{(i,j)\in\Delta(\mathsf{m}) }
    \vartheta_{\mathsf{m}_{i,j}+\mathsf{m}_{j,i}}(\mathsf{m}_{i,j}\odot x_i, \mathsf{m}_{j,i}\odot x_j),
}
where the set $\mathsf{M}_n(k_1,\dots,k_n)$ of nonnegative integer-valued $n\times n$ matrices defined by
 \begin{align*}
  \mathsf{M}_n(k_1,\dots,k_n)
   :=\bigcup_{k=1}^\infty\bigg\{& \mathsf{m}=(\mathsf{m}_{i,j})_{i,j=1}^n \in\ZZ_+^{n\times n}\colon \sum_{j\in\{1,\dots,n\}\setminus\{i\}} \mathsf{m}_{i,j} =k_i k,\; i\in\{1,\ldots,n\} \\
  &\hspace{4.3cm}\hbox{and }\ \#\Delta(\mathsf{m})\leq n-1 \bigg\}.
 \end{align*}
\end{Thm}

\begin{proof}
Note that the diagonal entries $\mathsf{m}_{i,i}$, $i\in\{1,\ldots,n\}$ for a matrix $\mathsf{m}\in \mathsf{M}_n$
 do not play any role in any of the formulae of the theorem (since $\mathsf{m}_{i,i}$ is excluded in the summation
 $\sum_{j\in\{1,\dots,n\}\setminus\{i\}}$),
 so without loss of generality we can and do assume that $\mathsf{m}_{i,i}=0$, $i\in\{1,\ldots,n\}$.
Consequently, for each $n\geq 2$, $n\in\NN$, $k_1,\dots,k_n\in\NN$, and $\mathsf{m}\in \mathsf{M}_n(k_1,\dots,k_n)$,
the matrix $\mathsf{m}$ can have at most $2n-2$ non-zero entries.

{\sl Step 1.} Let $n\geq 2$, $n\in\NN$, $x_1,\ldots,x_n\in X$ and $k_1,\ldots,k_n\in\NN$ be fixed arbitrarily.
Note that, for each $k\in\NN$, we have $\sign(\sum_{i=1}^n k_i\psi(x_i,t)) = \sign(\sum_{i=1}^n k_ik\psi(x_i,t))$, $t\in\Theta$, and hence
 \begin{align*}
 \vartheta_{k_1+\cdots+k_n}(k_1\odot x_1,\ldots,k_n\odot x_n)
 =\vartheta_{(k_1+\cdots+k_n)k}( (k_1k)\odot x_1,\ldots,(k_nk)\odot x_n).
 \end{align*}
Consequently, for each $k\in\NN$ and $\mathsf{m}=(\mathsf{m}_{i,j})_{i,j=1}^n\in\mathsf{M}_n(k_1,\ldots,k_n)$ with
 $\sum_{j\in\{1,\dots,n\}\setminus\{i\}} \mathsf{m}_{i,j} =k_ik$, $i=1,\ldots,n$, using that $\psi\in\Psi[T](X,\Theta)$,
 Theorem~\ref{Thm_psi_becsles_mean_prop} implies that
 \begin{align}\label{help_18}
  \begin{split}
 \vartheta_{k_1+\cdots+k_n}(k_1\odot x_1,\ldots,k_n\odot x_n)
    & = \vartheta_{(k_1+\cdots+k_n)k}( (k_1k)\odot x_1,\ldots,(k_nk)\odot x_n)  \\
    &\geq \min_{(i,j)\in\Delta(\mathsf{m})}
     \vartheta_{\mathsf{m}_{i,j}+\mathsf{m}_{j,i}}(\mathsf{m}_{i,j}\odot x_i, \mathsf{m}_{j,i}\odot x_j).
  \end{split}
 \end{align}
Indeed, for each $\alpha\in\{1,\ldots,n\}$, we have
 \begin{align*}
  &\big\{ \mathsf{m}_{i,j} : 1\leq i < j \leq n, \; i=\alpha \big\}
      \cup \big\{ \mathsf{m}_{j,i} : 1\leq i < j \leq n, \; j=\alpha \big\} \\
  &\qquad = \big\{ \mathsf{m}_{\alpha,j} : j=\alpha+1,\ldots,n \}
        \cup \big\{ \mathsf{m}_{\alpha,i} : i=1,\ldots,\alpha -1 \},
 \end{align*}
and $\sum_{j=\alpha+1}^n \mathsf{m}_{\alpha,j} + \sum_{i=1}^{\alpha-1} \mathsf{m}_{\alpha,i} = k_\alpha k$.
Therefore, the concatenation of the tuples $(\mathsf{m}_{i,j}\odot x_i, \mathsf{m}_{j,i}\odot x_j)$ for all $(i,j)\in\Delta(\mathsf{m})$
is equal to a permutation of the tuple $( (k_1k)\odot x_1,\ldots,(k_nk)\odot x_n)$.
This, due to Theorem~\ref{Thm_psi_becsles_mean_prop}, imply that the inequality in \eqref{help_18} is valid.
Consequently, the left hand side of \eqref{help_psi_2sample_formula_w} is greater than or equal to its right hand side.
Note that the above argument also shows that the left hand side of \eqref{help_psi_2sample_formula_w}
 is greater than or equal to its right hand side by taking the supremum over all the matrices
 $\mathsf{m} = (\mathsf{m}_{i,j})_{i,j=1}^n \in\ZZ_+^{n\times n}$ that satisfy
 $\sum_{j\in\{1,\dots,n\}\setminus\{i\}} \mathsf{m}_{i,j} =k_i k$, $i\in\{1,\ldots,n\}$, with some $k\in\NN$
 (i.e., without assuming that $\#\Delta(\mathsf{m})\leq n-1$ holds) instead of over $\mathsf{M}_n(k_1,\dots,k_n)$.

{\sl Step 2.}
It remains to check that the left hand side of \eqref{help_psi_2sample_formula_w} cannot be strictly larger than its right hand side.
In case of $n=2$, it holds trivially, since
 \[
 \mathsf{M}_2(k_1,k_2)
  = \left\{  \begin{bmatrix}
                0 & kk_1 \\
                kk_2 & 0 \\
            \end{bmatrix}: \quad k\in\NN
    \right\},
 \]
and $\vartheta_{k(k_1+k_2)}((kk_1)\odot x_1, (kk_2)\odot x_2) = \vartheta_{k_1+k_2}(k_1\odot x_1, k_2\odot x_2)$, $k\in\NN$.
Without loss of generality we can and do assume that $n\geq 3$ and $\vartheta_{1}(x_1) \leq \vartheta_{1}(x_2) \leq \cdots \leq \vartheta_{1}(x_n)$.
In what follows, we prove an intermediate formula for $\vartheta_{k_1+\dots+k_n}(k_1\odot x_1,\dots,k_n\odot x_n)$, which
 can be well-used in Step 3 for proving \eqref{help_psi_2sample_formula_w} by induction with respect to $n$.
Namely, we now show that
 \Eq{help_19}{
  &\vartheta_{k_1+\dots+k_n}(k_1\odot x_1,\dots,k_n\odot x_n)\\
  &=\sup_{\big\{q, p_1, p_n \in\NN: \, \min(qk_1-p_1, qk_n-p_n)=0 \big\}}\min \Big( \vartheta_{p_1+p_n}(p_1\odot x_1,p_n\odot x_n),\\
  &\hspace{1cm}\vartheta_{(k_1+\dots+k_n)q - (p_1+p_n)}
      \big((qk_1-p_1)\odot x_1, (qk_2)\odot x_2,\dots, (q k_{n-1})\odot x_{n-1}, (qk_n-p_n)\odot x_n\big) \Big).
 }
Using that $\psi\in\Psi[T](X,\Theta)$, Theorem~\ref{Thm_psi_becsles_mean_prop} and the equality
 \[
   \vartheta_{k_1+\dots+k_n}(k_1\odot x_1,\dots,k_n\odot x_n)
     = \vartheta_{(k_1+\dots+k_n)q}((qk_1)\odot x_1,\dots,(qk_n)\odot x_n),
   \qquad q\in\NN,
 \]
we get that the left hand side of \eqref{help_19} is greater than or equal to its right hand side.

It remains to check that the left hand side of \eqref{help_19} is less than or equal to its right hand side.
Using that $\psi\in\Psi[T](X,\Theta)$ and the assumption $\vartheta_{1}(x_1) \leq \vartheta_{1}(x_2) \leq \cdots \leq \vartheta_{1}(x_n)$,
 Theorem~\ref{Thm_psi_becsles_mean_prop} yields that
 \Eq{*}{
     \vartheta_1(x_1)&=\min(\vartheta_1(x_1),\dots,\vartheta_1(x_n))
    = \min(\vartheta_{k_1}(k_1\odot x_1),\ldots,\vartheta_{k_n}(k_n\odot x_n)) \\
    &\leq  \vartheta_{k_1+\dots+k_n}(k_1\odot x_1,\dots,k_n\odot x_n)
    \leq \max(\vartheta_{k_1}(k_1\odot x_1),\ldots,\vartheta_{k_n}(k_n\odot x_n))\\
          &= \max(\vartheta_1(x_1),\dots,\vartheta_1(x_n)) =\vartheta_1(x_n).
 }
We now consider two cases according to $\vartheta_1(x_1) = \vartheta_{k_1+\dots+k_n}(k_1\odot x_1,\dots,k_n\odot x_n)$
 and $\vartheta_1(x_1) < \vartheta_{k_1+\dots+k_n}(k_1\odot x_1,\dots,k_n\odot x_n)$.

As the first case, suppose that $\vartheta_1(x_1) = \vartheta_{k_1+\dots+k_n}(k_1\odot x_1,\dots,k_n\odot x_n)$.
Then, using again $\psi\in\Psi[T](X,\Theta)$, Theorem~\ref{Thm_psi_becsles_mean_prop} and that $\vartheta_1(x_i)\geq \vartheta_1(x_1)$, $i=1,\ldots,n$, we have
 \begin{align*}
   &\vartheta_{(k_1+\dots+k_n)q - (p_1+p_n)}
                    \big((qk_1-p_1)\odot x_1, (qk_2)\odot x_2,\dots, (q k_{n-1})\odot x_{n-1}, (qk_n-p_n)\odot x_n\big) \\
   &\qquad  \geq \vartheta_1(x_1) = \vartheta_{k_1+\dots+k_n}(k_1\odot x_1,\dots,k_n\odot x_n),
 \end{align*}
 and
 \[
   \vartheta_{p_1+p_n}(p_1\odot x_1,p_n\odot x_n)
       \geq \vartheta_1(x_1) = \vartheta_{k_1+\dots+k_n}(k_1\odot x_1,\dots,k_n\odot x_n)
 \]
 for all $q,p_1,p_n\in\NN$ satisfying $qk_1-p_1,qk_n-p_n\in\ZZ_+$.
This yields that, in this case, the right hand side of \eqref{help_19} is greater than equal to its left hand side, as desired.

As the second case, we can assume that $\vartheta_1(x_1) < \vartheta_{k_1+\cdots+k_n}(k_1\odot x_1,\ldots,k_n\odot x_n)$.
Choose a sufficiently small $\vare>0$ such that
 \[
   \vartheta_1(x_1) < \vartheta_{k_1+\cdots+k_n}(k_1\odot x_1,\ldots,k_n\odot x_n) - \vare.
 \]
Using that $\psi\in\Psi[T,W_2](X,\Theta)$, by Lemma~\ref{Lem_infinit}, the $\psi$-estimator $\vartheta_\psi:\bigcup_{n=1}^\infty X^n\to\Theta$,
 \[
    \vartheta_\psi(x_1,\ldots,x_n):= \vartheta_{n,\psi}(x_1,\ldots,x_n) = \vartheta_n(x_1,\ldots,x_n),
     \qquad n\in\NN,\quad x_1,\ldots,x_n\in X,
 \]
 is sensitive, and since
 \[
 \vartheta_1(x_1) < \vartheta_{k_1+\cdots+k_n}(k_1\odot x_1,\ldots,k_n\odot x_n) - \vare
                         < \vartheta_{k_1+\cdots+k_n}(k_1\odot x_1,\ldots,k_n\odot x_n)
                         \leq \vartheta_1(x_n),
 \]
 there exists $r_1,r_n\in\NN$ such that
 \begin{align}\label{help_comp5_uj}
  \begin{split}
   &\vartheta_{k_1+\cdots+k_n}(k_1\odot x_1,\ldots,k_n\odot x_n) - \vare   \\
   &\qquad  < \vartheta_{r_1+r_n}(r_1\odot x_1,r_n\odot x_n)
     < \vartheta_{k_1+\cdots+k_n}(k_1\odot x_1,\ldots,k_n\odot x_n).
  \end{split}
 \end{align}
In the case $\frac{k_1}{r_1}\leq \frac{k_n}{r_n}$, let
\Eq{*}{
  q:=r_1,\qquad p_1:=k_1r_1,\qquad p_n:=k_1r_n,
}
while, in the case $\frac{k_1}{r_1}>\frac{k_n}{r_n}$, let
\Eq{*}{
  q:=r_n,\qquad p_1:=k_nr_1,\qquad p_n:=k_nr_n.
}
Then $q,p_1,p_n\in\NN$ and $\min(qk_1-p_1, qk_n-p_n)=0$.
Using that $\psi\in\Psi[T](X,\Theta)$, Theorem~\ref{Thm_psi_becsles_mean_prop} and the symmetry of $\vartheta_n$, $n\in\NN$ (see the paragraph after Definition \ref{Def_Tn}), we have
 \begin{align}\label{help_comp8_uj}
  \begin{split}
  &\vartheta_{k_1+\cdots+k_n}(k_1\odot x_1,\ldots,k_n\odot x_n)
    = \vartheta_{q(k_1+\cdots+k_n)}((qk_1)\odot x_1,\ldots,(qk_n)\odot x_n) \\
  & \leq \max\Big(  \vartheta_{q(k_1+\cdots+k_n)-(p_1+p_n)}((qk_1-p_1)\odot x_1,(qk_2)\odot x_2,\ldots,(qk_{n-1})\odot x_{n-1},(qk_n-p_n)\odot x_n),\\
  &\phantom{\leq \max\Big(\;}
     \vartheta_{p_1+p_n}(p_1\odot x_1,p_n\odot x_n) \Big),
 \end{split}
 \end{align}
 and
 \begin{align}\label{help_comp9_uj}
  \begin{split}
  &\vartheta_{k_1+\cdots+k_n}(k_1\odot x_1,\ldots,k_n\odot x_n)
    = \vartheta_{q(k_1+\cdots+k_n)}((qk_1)\odot x_1,\ldots,(qk_n)\odot x_n) \\
  & \geq \min\Big(  \vartheta_{q(k_1+\cdots+k_n)-(p_1+p_n)}((qk_1-p_1)\odot x_1,(qk_2)\odot x_2,\ldots,(qk_{n-1})\odot x_{n-1},(qk_n-p_n)\odot x_n),\\
  &\phantom{\leq \max\Big(\;}
     \vartheta_{p_1+p_n}(p_1\odot x_1,p_n\odot x_n) \Big).
 \end{split}
 \end{align}
Concerning the inequalities \eqref{help_comp8_uj} and \eqref{help_comp9_uj}, we note that even if $qk_1-p_1 = qk_n - p_n=0$, then the tuple
 $((qk_1-p_1)\odot x_1,(qk_2)\odot x_2,\ldots,(qk_{n-1})\odot x_{n-1},(qk_n-p_n)\odot x_n)$ is not the empty tuple
 due to $n\geq 3$.
By the second inequality in \eqref{help_comp5_uj}, the symmetry of $\vartheta_n$, $n\in\NN$, and observing that $(p_1,p_n)=(k_ir_1,k_ir_n)$ holds either for $i=1$ or for $i=n$, we get
 \begin{align*}
 \vartheta_{p_1+p_n}(p_1\odot x_1,p_n\odot x_n)
 =\vartheta_{r_1+r_n}(r_1\odot x_1,r_n\odot x_n)
         < \vartheta_{k_1+\cdots+k_n}(k_1\odot x_1,\ldots,k_n\odot x_n).
 \end{align*}
Consequently, the maximum at the right hand side of the inequality \eqref{help_comp8_uj} is equal to
 \[
  \vartheta_{q(k_1+\cdots+k_n)-(p_1+p_n)}((qk_1-p_1)\odot x_1,(qk_2)\odot x_2,\ldots,(qk_{n-1})\odot x_{n-1},(qk_n-p_n)\odot x_n),
 \]
 since otherwise we would arrive at a contradiction.
This yields that the minimum at the right hand side of the inequality \eqref{help_comp9_uj} is equal to
 \[
   \vartheta_{p_1+p_n}(p_1\odot x_1,p_n\odot x_n).
 \]

All in all, using that $(p_1,p_n)=(k_ir_1,k_ir_n)$ holds either for $i=1$ or for $i=n$, and the first inequality in \eqref{help_comp5_uj}, for all sufficiently small $\vare>0$,
there exist $q,p_1,p_n\in\NN$ such that $\min(qk_1-p_1, qk_n-p_n)=0$ and
 \begin{align*}
  &\min\Big( \vartheta_{(k_1+\dots+k_n)q - (p_1+p_n)}
  \big((qk_1-p_1)\odot x_1, (qk_2)\odot x_2,\dots, (q k_{n-1})\odot x_{n-1}, (qk_n-p_n)\odot x_n\big),\\
  &\phantom{\min\Big\{\,} \vartheta_{p_1+p_n}(p_1\odot x_1,p_n\odot x_n) \Big)\\
  &= \vartheta_{p_1+p_n}(p_1\odot x_1,p_n\odot x_n)
   = \vartheta_{r_1+r_n}(r_1\odot x_1,r_n\odot x_n)
   > \vartheta_{k_1+\cdots+k_n}(k_1\odot x_1,\ldots,k_n\odot x_n) - \vare.
 \end{align*}
This implies that the right hand side of \eqref{help_19} is greater than or equal to its left hand side minus $\vare$,
 and, by taking the limit $\vare\downarrow 0$, we get that the right hand side of \eqref{help_19} is greater than or equal
 to its left hand side, as desired.

{\sl Step 3.}
Using Steps 1 and 2, we are going to prove \eqref{help_psi_2sample_formula_w} by induction with respect to $n\geq 3$.
Let us suppose that \eqref{help_psi_2sample_formula_w} holds for $2,\ldots,n-1$, and show that it holds for $n$ as well.
Let $\vare>0$, $x_1,\ldots,x_n\in X$ and $k_1,\ldots,k_n\in\NN$.
By Step 2, there exist $q,p_1,p_n\in\NN$ such that $\min(qk_1-p_1,qk_n-p_n) = 0$, and the following two inequalities hold
 \begin{align}\label{help_26}
  \begin{split}
 &\vartheta_{(k_1+\dots+k_n)q - (p_1+p_n)}
                    \big((qk_1-p_1)\odot x_1, (qk_2)\odot x_2,\dots, (q k_{n-1})\odot x_{n-1}, (qk_n-p_n)\odot x_n\big)\\
 &\qquad > \vartheta_{k_1+\cdots+k_n}(k_1\odot x_1,\ldots,k_n\odot x_n) - \vare,
  \end{split}
 \end{align}
 and
 \begin{align}\label{help_20}
 \vartheta_{p_1+p_n}(p_1\odot x_1,p_n\odot x_n) > \vartheta_{k_1+\cdots+k_n}(k_1\odot x_1,\ldots,k_n\odot x_n) - \vare.
 \end{align}
We can distinguish three cases, namely,
\begin{center}
 \begin{tabular}{ll}
   Case I: & $qk_n-p_n = 0$, $qk_1-p_1 \ne 0$, \\
   Case II: & $qk_n-p_n \ne 0$, $qk_1-p_1 = 0$, \\
   Case III: & $qk_n-p_n = 0$, $qk_1-p_1= 0$.
 \end{tabular}
\end{center}
In each of these cases, we are going to show that there exists $\widehat{\mathsf{m}}=(\widehat{\mathsf{m}}_{i,j})_{i,j=1}^n\in M_n(k_1,\ldots, k_n)$ such that
 \Eq{help_20.5}{
  \min_{(i,j)\in\Delta(\widehat{\mathsf{m}})}
    \vartheta_{\widehat{\mathsf{m}}_{i,j}+\widehat{\mathsf{m}}_{j,i}}( \widehat{\mathsf{m}}_{i,j}\odot x_i,\widehat{\mathsf{m}}_{j,i}\odot x_j)
      > \vartheta_{k_1+\cdots+k_n}(k_1\odot x_1,\ldots,k_n\odot x_n) - 2\vare.
 }

In Case I, by the induction hypothesis, there exist $\mathsf{m}=(\mathsf{m}_{i,j})_{i,j=1}^{n-1} \in\ZZ_+^{(n-1)\times (n-1)}$ and $k\in\NN$ such that
\Eq{help_20a}{
  \sum_{j\in\{1,\dots,n-1\}\setminus\{1\}} \mathsf{m}_{1,j} =k(qk_1-p_1),
   \quad \sum_{j\in\{1,\dots,n-1\}\setminus\{i\}} \mathsf{m}_{i,j} = k(qk_i),\qquad i\in\{2,\ldots,n-1\},
}
\Eq{help_20b}{
  \#\Delta(\mathsf{m})\leq n-2,
}
and
\Eq{help_21}{
   &\min_{(i,j)\in\Delta(\mathsf{m})}
    \vartheta_{\mathsf{m}_{i,j}+\mathsf{m}_{j,i}}(\mathsf{m}_{i,j}\odot x_i, \mathsf{m}_{j,i}\odot x_j)\\
      &\hspace{18mm}> \vartheta_{(k_1+\cdots+k_{n-1})q-p_1}((qk_1-p_1)\odot x_1,(qk_2)\odot x_2\ldots,(qk_{n-1})\odot x_{n-1}) - \vare\\
      &\hspace{18mm}> \vartheta_{k_1+\cdots+k_n}(k_1\odot x_1,\ldots,k_n\odot x_n) - 2\vare,
}
where the last inequality is a consequence of $\eqref{help_26}$.
Let $\widehat{\mathsf{m}}=(\widehat{\mathsf{m}}_{i,j})_{i,j=1}^n \in\ZZ_+^{n\times n}$ be defined by
 \[
   \widehat{\mathsf{m}} = \begin{bmatrix}
                            \mathsf{m}_{1,1} & \mathsf{m}_{1,2}  & \ldots & \mathsf{m}_{1,n-1} & kp_1 \\
                            \mathsf{m}_{2,1} & \mathsf{m}_{2,2} & \ldots & \mathsf{m}_{2,n-1}  & 0 \\
                            \vdots &  \vdots & \ddots &  \vdots & \vdots \\
                            \mathsf{m}_{n-1,1} & \mathsf{m}_{n-1,2} & \ldots & \mathsf{m}_{n-1,n-1} & 0 \\
                            kp_n &  0 & \ldots & 0 & 0 \\
                          \end{bmatrix}.
 \]
Then, using \eqref{help_20a} and that $kp_n=k(qk_n)$, we have
 \begin{align}\label{help_22}
  \sum_{j\in\{1,\dots,n\}\setminus\{i\}} \widehat{\mathsf{m}}_{i,j} =k(qk_i),\qquad i=1,\ldots, n,
 \end{align}
and, applying \eqref{help_20b} and that $\widehat{\mathsf{m}}_{1,n} + \widehat{\mathsf{m}}_{n,1} = k(p_1+p_n)\geq 1$,
 we can see that
 \[
 \#\Delta(\widehat{\mathsf{m}})
 =\#\Delta(\mathsf{m})+1\leq n-1.
 \]
Therefore, it follows that $\widehat{\mathsf{m}}\in M_n(k_1,\ldots,k_n)$.
By \eqref{help_20} and \eqref{help_21}, we obtain
 \Eq{*}{
  \min_{(i.j)\in\Delta(\widehat{\mathsf{m}})}
    &\vartheta_{\widehat{\mathsf{m}}_{i,j}+\widehat{\mathsf{m}}_{j,i}}( \widehat{\mathsf{m}}_{i,j}\odot x_i,\widehat{\mathsf{m}}_{j,i}\odot x_j)\\
    &=\min\bigg(
    \vartheta_{k(p_1+p_n)}((kp_1)\odot x_1,(kp_n)\odot x_n), \min_{(i,j)\in\Delta(\mathsf{m})}
    \vartheta_{\mathsf{m}_{i,j}+\mathsf{m}_{j,i}}(\mathsf{m}_{i,j}\odot x_i, \mathsf{m}_{j,i}\odot x_j)\bigg)\\
    &=\min\bigg(
    \vartheta_{p_1+p_n}( p_1\odot x_1, p_n\odot x_n), \min_{(i,j)\in\Delta(\mathsf{m})}
    \vartheta_{\mathsf{m}_{i,j}+\mathsf{m}_{j,i}}(\mathsf{m}_{i,j}\odot x_i, \mathsf{m}_{j,i}\odot x_j)\bigg)\\
      &> \vartheta_{k_1+\cdots+k_n}(k_1\odot x_1,\ldots,k_n\odot x_n) - 2\vare.
 }
This inequality shows that \eqref{help_20.5} holds in Case I.

In Case II, by the induction hypothesis, there exist $\mathsf{m}=(\mathsf{m}_{i,j})_{i,j=2}^n \in\ZZ_+^{(n-1)\times (n-1)}$ and $k\in\NN$ such that
 \[
  \sum_{j\in\{2,\dots,n\}\setminus\{i\}} \mathsf{m}_{i,j} =k(qk_i),\qquad i=2,\ldots,n-1,
   \qquad \sum_{j\in\{2,\dots,n\}\setminus\{n\}} \mathsf{m}_{n,j} = k(qk_n-p_n),
 \]
 \[
  \#\Delta(\mathsf{m})\leq n-2,
 \]
 and
 \Eq{help_23}{
   &\min_{(i,j)\in\Delta(\mathsf{m})}
    \vartheta_{\mathsf{m}_{i,j}+\mathsf{m}_{j,i}}(\mathsf{m}_{i,j}\odot x_i, \mathsf{m}_{j,i}\odot x_j)\\
      &\hspace{18mm}> \vartheta_{(k_2+\cdots+k_n)q-p_n}((qk_2)\odot x_2,\ldots,(qk_{n-1})\odot x_{n-1},(qk_n-p_n)\odot x_n) - \vare\\
      &\hspace{18mm}>\vartheta_{k_1+\cdots+k_n}(k_1\odot x_1,\ldots,k_n\odot x_n) - 2\vare,
 }
where the last inequality is a consequence of $\eqref{help_26}$.
Let $\widehat{\mathsf{m}}=(\widehat{\mathsf{m}}_{i,j})_{i,j=1}^n \in\ZZ_+^{n\times n}$ be defined by
 \[
   \widehat{\mathsf{m}} = \begin{bmatrix}
                            0 & 0  & \ldots & 0 & kp_1 \\
                            0 & \mathsf{m}_{2,2} & \ldots & \mathsf{m}_{2,n-1}  & \mathsf{m}_{2,n} \\
                            \vdots &  \vdots & \ddots &  \vdots & \vdots \\
                             0 & \mathsf{m}_{n-1,2} & \ldots & \mathsf{m}_{n-1,n-1} & \mathsf{m}_{n-1,n} \\
                            kp_n &  \mathsf{m}_{n,2} & \ldots & \mathsf{m}_{n,n-1} & \mathsf{m}_{n,n} \\
                          \end{bmatrix}.
 \]
Then \eqref{help_22} holds since $kp_1=k(qk_1)$ and we also have that $\#\Delta(\widehat{\mathsf{m}})
 =\#\Delta(\mathsf{m})+1\leq n-1$. Therefore, $\widehat{\mathsf{m}}\in M_n(k_1,\ldots,k_n)$.
Using \eqref{help_20} and \eqref{help_23}, we can finish the proof just as in Case I, and consequently, \eqref{help_20.5} holds in Case II as well.

In Case III, by the induction hypothesis, there exist $\mathsf{m}=(\mathsf{m}_{i,j})_{i,j=2}^{n-1} \in\ZZ_+^{(n-2)\times (n-2)}$ and $k\in\NN$ such that
 \[
  \sum_{j\in\{2,\dots,n-1\}\setminus\{i\}} \mathsf{m}_{i,j} =k(qk_i),\qquad i\in\{2,\ldots,n-1\},
 \]
 \[
  \#\Delta(\mathsf{m})\leq n-3,
 \]
and
 \Eq{help_24}{
   \min_{(i,j)\in\Delta(\mathsf{m})}
    \vartheta_{\mathsf{m}_{i,j}+\mathsf{m}_{j,i}}(\mathsf{m}_{i,j}\odot x_i, \mathsf{m}_{j,i}\odot x_j)
      &> \vartheta_{q(k_2+\cdots+k_{n-1})}((qk_2)\odot x_2,\ldots,(qk_{n-1})\odot x_{n-1}) - \vare\\
      &>\vartheta_{k_1+\cdots+k_n}(k_1\odot x_1,\ldots,k_n\odot x_n) - 2\vare,
 }
where the last inequality is a consequence of $\eqref{help_26}$.
Let $\widehat{\mathsf{m}}=(\widehat{\mathsf{m}}_{i,j})_{i,j=1}^n \in\ZZ_+^{n\times n}$ be defined by
 \[
   \widehat{\mathsf{m}} = \begin{bmatrix}
                            0 & 0  & \ldots & 0 & kp_1 \\
                            0 & \mathsf{m}_{2,2} & \ldots & \mathsf{m}_{2,n-1}  & 0 \\
                            \vdots &  \vdots & \ddots &  \vdots & \vdots \\
                             0 & \mathsf{m}_{n-1,2} & \ldots & \mathsf{m}_{n-1,n-1} & 0 \\
                            kp_n &  0 & \ldots & 0 & 0 \\
                          \end{bmatrix}.
 \]
Then \eqref{help_22} holds since $kp_1=k(qk_1)$ and $kp_n=k(qk_n)$. Using that $\widehat{\mathsf{m}}_{1,n} + \widehat{\mathsf{m}}_{n,1} = k(p_1+p_n)\geq 1$, we also have that $ \#\Delta(\widehat{\mathsf{m}})=\#\Delta(\mathsf{m}) +1\leq n-2<n-1$, whence $\widehat{\mathsf{m}}\in M_n(k_1,\ldots, k_n)$ follows.
Using \eqref{help_20} and \eqref{help_24}, we can finish the proof of \eqref{help_20.5} just as in Case I,
and consequently, \eqref{help_20.5} holds in Case III as well.

It follows from the inequality \eqref{help_20.5} that the right hand side of
\eqref{help_psi_2sample_formula_w} is bigger than $\vartheta_{k_1+\cdots+k_n}(k_1\odot x_1,\ldots,k_n\odot x_n) - 2\vare$.
Since $\vare>0$ was arbitrary, by taking the limit $\vare\downarrow 0$, we get that the right hand side of
 \eqref{help_psi_2sample_formula_w} is greater than or equal to its left hand side.
In view of Step 1, this completes the proof of the equality in \eqref{help_psi_2sample_formula_w}.
\end{proof}

The following result is an immediate consequence of  Theorems \ref{Thm_2_to_n} and \ref{Thm_2_to_n+w}.

\begin{Cor}\label{Rem_determ_class}
The set $S$ defined by
\Eq{*}{
  S:=\{(x_1,\dots,x_n)\in X^n\mid n\in\NN,\, \#\{x_1,\dots,x_n\}\leq2\}
}
is a determining class for $\Psi[T,W_2](X,\Theta)$.
\end{Cor}

\begin{Rem}
If $\psi\in\Psi[C,W_2](\Theta,\Theta)$ with $\psi(t,t)=0$ for all $t\in\Theta$, then $\psi$ becomes a so-called quasi-deviation and the corresponding generalized $\psi$-estimator will be equal to a quasi-deviation mean (see P\'ales \cite{Pal82a}).
Hence, in this case, Theorem \ref{Thm_2_to_n} provides a formula for quasi-deviation means as well.
\proofend
\end{Rem}

Next, we present three applications of Theorem \ref{Thm_2_to_n}.
The first and second are related to the maximum likelihood estimator of some parameter of certain absolutely continuous random variables, and the third is about expectiles.

\begin{Ex}\label{Ex3}
Let $p\in\NN$ be an odd integer.
Let $f:\RR\times\RR\to\RR$,
 \[
   f(x,t):= \frac{1}{2(p+1)^{\frac{1}{p+1}} \Gamma(\frac{p+2}{p+1})}  \exp\bigg(-\frac{(x-t)^{p+1}}{p+1}\bigg),\qquad  x,t\in\RR.
 \]
Note that, for all $t\in\RR$, the function $\RR\ni x \mapsto f(x,t)$ is the density function of a random variable with generalized normal distribution having location parameter $t$, scale parameter $(p+1)^{\frac{1}{p+1}}$ and shape parameter $p+1$.
The location parameter $t$ coincides also with the mean of the distribution.
Generalized normal distributions have been frequently used in image processing and communication theory. Observe that for $p=1$ and $t=0$, the map $\RR\ni x\mapsto f(x,0)$ is the density function of the standard normal distribution.

Suppose that $p$ is known.
Then, as it was explained in Section 4 in Barczy and P\'ales \cite[pages 32 and 33]{BarPal2}, for each $n\in\NN$ and $x_1,\ldots,x_n\in\RR$, the equation
 $\sum_{i=1}^n \psi(x_i,t)=0$, $t\in\RR$, where $\psi:\RR\times\RR\to\RR$,
 \[
   \psi(x,t):=\frac{\partial_2f(x,t)}{f(x,t)} = (x-t)^p, \qquad x,t\in\RR,
 \]
 is nothing else but the likelihood equation based on the observations $x_1,\ldots,x_n$, when we estimate
 the parameter $t\in\RR$ using maximum likelihood estimation.
Then $\psi\in\Psi(\RR,\RR)$, and $\psi$ is a $Z_1$-function with $\vartheta_1(x)=x$, $x\in\RR$.
Further, for all $x,y\in\RR$ with $\vartheta_1(x)=x<y=\vartheta_1(y)$, the function
 \[
 (\vartheta_1(x),\vartheta_1(y)) = (x,y)\ni t \mapsto - \frac{\psi(x,t)}{\psi(y,t)} = - \frac{(x-t)^p}{(y-t)^p} = \left( \frac{y-x}{y-t} - 1 \right)^p
 \]
 is strictly increasing.
This is in accordance with part (i) of Proposition 2 in Barczy and P\'ales \cite{BarPal2}, since, for all $x\in\RR$, the function
 $\RR\ni t \mapsto \psi(x,t)=(x-t)^p$ is strictly decreasing.
Consequently, using part (vi) of Theorem 1 in Barczy and P\'ales \cite{BarPal2}, we have that  $\psi$ is a $W$-function,
(i.e., it is a $W_n$-function for each $n\in\NN$).
Taking into account that $\psi$ is continuous in its second variable, we have $\psi$ is a $Z_n^{\pmb{\lambda}}$-function for each $n\in\NN$ and $\pmb{\lambda}=(\lambda_1,\ldots,\lambda_n)\in\Lambda_n$, i.e., for each $n\in\NN$, $x_1,\ldots,x_n\in \RR$, and $\blambda\in\Lambda_n$, the equation
 \[
  \sum_{i=1}^n \lambda_i \psi(x_i,t) = \sum_{i=1}^n \lambda_i (x_i-t)^p = 0, \qquad t\in\RR,
 \]
has a unique solution, denoted by $\vartheta_n^{\pmb{\lambda}}(x_1,\ldots,x_n)$.
We note that, in general, this equation cannot be solved explicitly in an algebraic manner for $n\geq3$.
On the other hand, the situation is completely different if $n=2$.
Indeed, for $x_1,x_2\in\RR$ and $(\lambda_1,\lambda_2)\in\Lambda_2$, by solving the equation
 \[
   \lambda_1 \psi(x_1,t) + \lambda_2 \psi(x_2,t)=\lambda_1(x_1-t)^p + \lambda_2(x_2-t)^p=0, \qquad t\in\RR,
 \]
 one can get that
 \[
  \vartheta_2^{(\lambda_1,\lambda_2)}(x_1,x_2) = \frac{\sqrt[p]{\lambda_1}x_1 + \sqrt[p]{\lambda_2}x_2}{\sqrt[p]{\lambda_1} + \sqrt[p]{\lambda_2}}.
 \]
This implies that
 \begin{align}\label{help_33}
  \vartheta_{\ell_1+\ell_2}(\ell_1\odot x_1,\ell_2\odot x_2)
   = \vartheta_2^{(\ell_1,\ell_2)}(x_1,x_2) = \frac{\sqrt[p]{\ell_1}\,x_1 + \sqrt[p]{\ell_2}\,x_2}{\sqrt[p]{\ell_1} + \sqrt[p]{\ell_2}}
 \end{align}
 for all $x_1,x_2\in\RR$ and $\ell_1,\ell_2\in\ZZ_+$ with $\ell_1+\ell_2\geq 1$.
By Theorem \ref{Thm_2_to_n}, for all $n\geq 2$, $n\in\NN$ and $x_1,\ldots,x_n\in X$, we get
 \[
  \vartheta_{n}(x_1,\dots,x_n)
         =\sup_{\mathsf{m}\in \mathsf{M}_n}\min_{(i,j)\in\Delta(\mathsf{m})}
          \frac{\sqrt[p]{\mathsf{m}_{i,j}}\,x_i + \sqrt[p]{\mathsf{m}_{j,i}}\,x_j}{\sqrt[p]{\mathsf{m}_{i,j}} + \sqrt[p]{\mathsf{m}_{j,i}}},
 \]
where the matrix $\mathsf{M}_n$ and the set $\Delta(\mathsf{m})$, $\mathsf{m}\in \mathsf{M}_n$, are defined in Theorem \ref{Thm_2_to_n} and in \eqref{help_Delta_m}, respectively.
Recall that the diagonal entries $\mathsf{m}_{i,i}$, $i\in\{1,\ldots,n\}$ of a matrix $\mathsf{m}\in \mathsf{M}_n$
 do not play any role in the formula for $\vartheta_{n}(x_1,\dots,x_n)$, so without loss of generality we can and do assume that $\mathsf{m}_{i,i}=0$, $i\in\{1,\ldots,n\}$.

In the special case $n=3$, for each $\mathsf{m}=(\mathsf{m}_{i,j})_{i,j=1}^3 \in\ZZ_+^{3\times 3}$, we have
 \[
   \Delta(\mathsf{m})=\big\{(i,j)\colon 1\leq i<j \leq 3,\ \mathsf{m}_{i,j}+\mathsf{m}_{j,i} \geq 1 \big\},
 \]
and the inequality $\#\Delta(\mathsf{m})\leq 2$ holds if and only if at least one of the following three equalities is satisfied:
 $(\mathsf{m}_{1,2}, \mathsf{m}_{2,1}) = (0,0)$,
 $(\mathsf{m}_{1,3}, \mathsf{m}_{3,1}) = (0,0)$ or $(\mathsf{m}_{2,3}, \mathsf{m}_{3,2}) = (0,0)$ (equivalently, $\mathsf{m}_{1,2}+\mathsf{m}_{2,1}=0$, $\mathsf{m}_{1,3}+\mathsf{m}_{3,1}=0$ or $\mathsf{m}_{2,3}+\mathsf{m}_{3,2}=0$).
Consequently, we get
\Eq{*}{
  \mathsf{M}_3
    &= \bigcup_{k=1}^\infty
          \left\{
              \begin{bmatrix}
                0 & a & k-a \\
                k-b & 0 & b \\
                k-c & c & 0 \\
              \end{bmatrix}
              : a,b,c\in\{0,1,\ldots,k\},\ (k+a-b)(2k-a-c)(b+c)=0
         \right\}\\
  &= \bigcup_{k=1}^\infty
          \left\{
              \begin{bmatrix}
                0 & 0 & k \\
                0 & 0 & k \\
                k-\ell & \ell & 0 \\
              \end{bmatrix},
              \begin{bmatrix}
                0 & k & 0 \\
                k-\ell & 0 & \ell \\
                0 & k & 0 \\
              \end{bmatrix},
              \begin{bmatrix}
                0 & \ell & k-\ell \\
                k & 0 & 0 \\
                k & 0 & 0 \\
              \end{bmatrix}
              : \ell\in\{0,1,\ldots,k\}
         \right\} \\
  &=\bigcup_{\{u,v\in\ZZ_+:\, u+v\geq1\}}
  \left\{
    \begin{bmatrix}
  0 & u & v \\
  u+v & 0 & 0 \\
  u+v & 0 & 0 \\
  \end{bmatrix},\quad
  \begin{bmatrix}
  0 & u+v & 0 \\
  u & 0 & v \\
  0 & u+v & 0 \\
  \end{bmatrix},\quad
  \begin{bmatrix}
  0 & 0 & u+v \\
  0 & 0 & u+v \\
  u & v & 0 \\
  \end{bmatrix}\right\}.
}
Note that the set $M_3$ is a union of matrices of 3 different types. In principle, for a general $n\geq 2$, $n\in\NN$, the set $M_n$ could be written in a similar form (it would be a union of matrices of $\binom{n(n-1)/2}{n-1}$ different types).
This and Theorem \ref{Thm_2_to_n} yield that, for all $x_1,x_2,x_3\in\RR$, the unique solution of the equation
  \begin{align}\label{help_30}
    (x_1-t)^p + (x_2-t)^p + (x_3-t)^p=0,\qquad t\in\RR,
  \end{align}
can be written as
\begin{align}\nonumber
   &\vartheta_3(x_1,x_2,x_3)&&\\\label{help_29}
   &= \sup_{\{u,v\in\ZZ_+:\, u+v\geq1\}}\max\bigg\{
   &&\hspace{-5mm}\min\Big(\vartheta_{2u+v}(u\odot x_1,(u+v)\odot x_2),\vartheta_{u+2v}(v\odot x_1,(u+v)\odot x_3)\Big), \\
   &&&\hspace{-6mm}\min\Big(\vartheta_{2u+v}((u+v)\odot x_1,u\odot x_2),\vartheta_{u+2v}(v\odot x_2,(u+v)\odot x_3)\Big),\nonumber\\
   &&&\hspace{-6mm}\min\Big(\vartheta_{2u+v}((u+v)\odot x_1,u\odot x_3),\vartheta_{u+2v}((u+v)\odot x_2,v\odot x_3)\Big) \bigg\},\nonumber
 \end{align}
 that is,
 \begin{align}\label{help_31}
  \begin{split}
   \vartheta_3(x_1,x_2,x_3)
    = \sup_{\{u,v\in\ZZ_+:\, u+v\geq1\}}\max\bigg\{
    &\min\bigg(
    \frac{\sqrt[p]{u}\,x_1 + \sqrt[p]{u+v}\,x_2}{\sqrt[p]{u+v} + \sqrt[p]{u}},
    \frac{\sqrt[p]{v}\,x_1 + \sqrt[p]{u+v}\,x_3}{\sqrt[p]{u+v} + \sqrt[p]{v}}\bigg),\\
    &\min\bigg(
    \frac{\sqrt[p]{u+v}\,x_1 + \sqrt[p]{u}\,x_2}{\sqrt[p]{u+v} + \sqrt[p]{u}},
    \frac{\sqrt[p]{v}\,x_2 + \sqrt[p]{u+v}\,x_3}{\sqrt[p]{u+v} + \sqrt[p]{v}}\bigg),\\
    &\min\bigg(
    \frac{\sqrt[p]{u+v}\,x_1 + \sqrt[p]{u}\,x_3}{\sqrt[p]{u+v} + \sqrt[p]{u}},
    \frac{\sqrt[p]{u+v}\,x_2 + \sqrt[p]{v}\,x_3}{\sqrt[p]{u+v} + \sqrt[p]{v}}\bigg)
    \bigg\}.
  \end{split}
 \end{align}
{It is easy to see that, for any continuous function $h:[0,1]\to\RR$, it holds that
 \begin{align}\label{help_36}
   \sup_{\{u,v\in\ZZ_+:\, u+v\geq1\}} h\bigg( \frac{u}{u+v}\bigg) =  \sup_{\lambda\in(0,1)\cap\QQ} h(\lambda)
   =  \sup_{\lambda\in[0,1]} h(\lambda).
 \end{align}
Consequently, using also \eqref{help_31}, we get that
 \begin{align}\label{help_32}
  \begin{split}
   \vartheta_3(x_1,x_2,x_3)
    = \sup_{\lambda\in[0,1]}
    \max\bigg\{
    &\min\bigg(
    \frac{\sqrt[p]{\lambda}\,x_1 + x_2}{1 + \sqrt[p]{\lambda}},
    \frac{\sqrt[p]{1-\lambda}\,x_1 + x_3}{1 + \sqrt[p]{1-\lambda}}\bigg),\\
    &\min\bigg(
    \frac{x_1 + \sqrt[p]{\lambda}\,x_2 }{1 + \sqrt[p]{\lambda}},
    \frac{ \sqrt[p]{1-\lambda}\,x_2 + x_3}{1 + \sqrt[p]{1-\lambda}}\bigg),\\
    &\min\bigg(
    \frac{x_1 + \sqrt[p]{\lambda}\, x_3}{1 + \sqrt[p]{\lambda}},
    \frac{x_2 + \sqrt[p]{1-\lambda}\, x_3}{1 + \sqrt[p]{1-\lambda}}\bigg)
    \bigg\}.
 \end{split}
 \end{align}

Next, as for demonstration, we consider some special choices of $x_i$, $i=1,2,3$.

If $x_1=x_2=x_3=1$, then the equation \eqref{help_30} reduces to $(1-t)^p=0$, $t\in\RR$,
 which has the unique solution $t=1$, i.e., $\vartheta_3(1,1,1)=1$.
This is in accordance with \eqref{help_32}, since, in the considered special case, the right-hand side of \eqref{help_32}
 is equal to $\sup_{\lambda\in[0,1]} \{\min(1,1)\}=1$.}

If $x_1=x_2=1$ and $x_3=0$, then the equation \eqref{help_30} reduces to $2(1-t)^p - t^p=0$, $t\in\RR$, which has the unique solution
 \[
   \vartheta_3(1,1,0) = \frac{\sqrt[p]{2}}{1+\sqrt[p]{2}}.
 \]
Note that $\vartheta_3(1,1,0) = \vartheta_2^{(2,1)}(1,0)$, and hence \eqref{help_33} also yields the above expression.
Next, we calculate $\vartheta_3(1,1,0)$ using \eqref{help_32} as well.
In the considered special case, the right-hand side of \eqref{help_32} is equal to
 \begin{align*}
   &\sup_{\lambda\in[0,1]}\max
    \bigg\{
    \min\Bigg( 1, \frac{\sqrt[p]{1-\lambda}}{1 + \sqrt[p]{1-\lambda}}\bigg),
    \min\bigg( 1, \frac{\sqrt[p]{1-\lambda}}{1 + \sqrt[p]{1-\lambda}}\bigg),
    \min\bigg(
    \frac{1}{1 + \sqrt[p]{\lambda}}, \frac{1}{1 + \sqrt[p]{1-\lambda}}\bigg)
    \bigg\}\\
   &\qquad = \sup_{\lambda\in[0,1]}\max
    \bigg\{ \frac{\sqrt[p]{1-\lambda}}{1 + \sqrt[p]{1-\lambda}},
            \min\bigg(
    \frac{1}{1 + \sqrt[p]{\lambda}}, \frac{1}{1 + \sqrt[p]{1-\lambda}}\bigg)
     \bigg\}.
 \end{align*}
If $\lambda\in[0,\frac{1}{2}]$, then $1/(1 + \sqrt[p]{1-\lambda})\leq 1/(1+\sqrt[p]{\lambda})$, and hence
 \begin{align*}
  &\sup_{\lambda\in[0,\frac{1}{2}]}\max
    \bigg\{ \frac{\sqrt[p]{1-\lambda}}{1 + \sqrt[p]{1-\lambda}},
            \min\bigg(
    \frac{1}{1 + \sqrt[p]{\lambda}}, \frac{1}{1 + \sqrt[p]{1-\lambda}}\bigg)
     \bigg\} \\
 &\qquad =  \sup_{\lambda\in[0,\frac{1}{2}]}\max
    \bigg\{ \frac{\sqrt[p]{1-\lambda}}{1 + \sqrt[p]{1-\lambda}},
             \frac{1}{1 + \sqrt[p]{1-\lambda}}
     \bigg\}
  = \sup_{\lambda\in[0,\frac{1}{2}]}
      \frac{1}{1 + \sqrt[p]{1-\lambda}}
 = \frac{1}{1 + \sqrt[p]{\frac{1}{2}}}
  = \frac{\sqrt[p]{2}}{1+\sqrt[p]{2}}.
 \end{align*}
If $\lambda\in[\frac{1}{2},1]$, then $1/(1 + \sqrt[p]{\lambda})\leq 1/(1+\sqrt[p]{1-\lambda})$, and hence
 \begin{align*}
  &\sup_{\lambda\in[\frac{1}{2},1]}\max
    \bigg\{ \frac{\sqrt[p]{1-\lambda}}{1 + \sqrt[p]{1-\lambda}},
            \min\bigg(
    \frac{1}{1 + \sqrt[p]{\lambda}}, \frac{1}{1 + \sqrt[p]{1-\lambda}}\bigg)
     \bigg\} \\
 &\qquad =  \sup_{\lambda\in[\frac{1}{2},1]}\max
    \bigg\{ \frac{\sqrt[p]{1-\lambda}}{1 + \sqrt[p]{1-\lambda}},\frac{1}{1 + \sqrt[p]{\lambda}}
     \bigg\}
  = \max
    \bigg\{ \sup_{\lambda\in[\frac{1}{2},1]}\frac{\sqrt[p]{1-\lambda}}{1 + \sqrt[p]{1-\lambda}},\sup_{\lambda\in[\frac{1}{2},1]}\frac{1}{1 + \sqrt[p]{\lambda}}
     \bigg\}
 \end{align*}
 \begin{align*}
 &\qquad
  = \max\bigg\{ \frac{1}{1 + \sqrt[p]{2}}, \frac{\sqrt[p]{2}}{1 + \sqrt[p]{2}} \bigg\}
  = \frac{\sqrt[p]{2}}{1+\sqrt[p]{2}}.
 \end{align*}
This yields that
 \begin{align*}
   \sup_{\lambda\in[0,1]}\max
    \bigg\{ \frac{\sqrt[p]{1-\lambda}}{1 + \sqrt[p]{1-\lambda}},\min\bigg(
    \frac{1}{1 + \sqrt[p]{\lambda}}, \frac{1}{1 + \sqrt[p]{1-\lambda}}\bigg)
     \bigg\}
      = \frac{\sqrt[p]{2}}{1+\sqrt[p]{2}},
 \end{align*}
 as desired.

If $x_1:=x$, $x_2:=\frac{x+y}{2}$ and $x_3:=y$, where $x,y\in\RR$  with $x<y$ are fixed, then the equation \eqref{help_30} reduces to
 \begin{align}\label{help_35}
  (x-t)^p + \Big(\frac{x+y}{2} - t \Big)^p + (y-t)^p = 0,\qquad t\in\RR,
 \end{align}
 which has the unique solution
 \[
   \vartheta_3\Big(x,\frac{x+y}{2},y\Big)
   = \frac{x+y}{2}.
 \]
Note that the unique solution of \eqref{help_35} is nothing else but the arithmetic mean of $x$ and $y$
 for any odd integer $p\in\NN$.
Next, we calculate $\vartheta_3\left(x,\frac{x+y}{2},y\right)$ using \eqref{help_32}.
In the considered special case, the right-hand side of \eqref{help_32} is equal to
 \begin{align*}
   \sup_{\lambda\in[0,1]}\max
    \bigg\{
    &\min\bigg(
    \frac{\sqrt[p]{\lambda}\,x + \frac{x+y}{2}}{1 + \sqrt[p]{\lambda}},
    \frac{\sqrt[p]{1-\lambda}\,x + y}{1 + \sqrt[p]{1-\lambda}}\bigg),\\
    &\min\bigg(
    \frac{x + \sqrt[p]{\lambda}\,\frac{x+y}{2} }{1 + \sqrt[p]{\lambda}},
    \frac{ \sqrt[p]{1-\lambda}\,\frac{x+y}{2} + y}{1 + \sqrt[p]{1-\lambda}}\bigg),\\
    &\min\bigg(
    \frac{x + \sqrt[p]{\lambda}\, y}{1 + \sqrt[p]{\lambda}},
    \frac{\frac{x+y}{2} + \sqrt[p]{1-\lambda}\, y}{1 + \sqrt[p]{1-\lambda}}\bigg)
    \bigg\}.
 \end{align*}
One can easily check that the inequality
 \begin{align*}
   \frac{\sqrt[p]{\lambda}\,x + \frac{x+y}{2}}{1 + \sqrt[p]{\lambda}}
   \leq \frac{\sqrt[p]{1-\lambda}\,x + y}{1 + \sqrt[p]{1-\lambda}}
 \end{align*}
 is equivalent to
 \[
   0 \leq \sqrt[p]{\lambda}(y-x) + (1-\sqrt[p]{1-\lambda})\frac{y-x}{2},
 \]
 which is satisfied, since $\lambda\in[0,1]$ and $x<y$.
This yields that
 \[
 \min\Bigg(
    \frac{\sqrt[p]{\lambda}\,x + \frac{x+y}{2}}{1 + \sqrt[p]{\lambda}},
    \frac{\sqrt[p]{1-\lambda}\,x + y}{1 + \sqrt[p]{1-\lambda}}\Bigg)
   = \frac{\sqrt[p]{\lambda}\,x + \frac{x+y}{2}}{1 + \sqrt[p]{\lambda}},
   \qquad \lambda\in[0,1].
 \]
Similarly, one can check that
 \begin{align*}
  \min\bigg(
    \frac{x + \sqrt[p]{\lambda}\,\frac{x+y}{2} }{1 + \sqrt[p]{\lambda}},
    \frac{ \sqrt[p]{1-\lambda}\,\frac{x+y}{2} + y}{1 + \sqrt[p]{1-\lambda}}\bigg)
   =  \frac{x + \sqrt[p]{\lambda}\,\frac{x+y}{2} }{1 + \sqrt[p]{\lambda}},
   \qquad \lambda\in[0,1],
 \end{align*}
 and
 \begin{align*}
 \min\bigg(
    \frac{x + \sqrt[p]{\lambda}\, y}{1 + \sqrt[p]{\lambda}},
    \frac{\frac{x+y}{2} + \sqrt[p]{1-\lambda}\, y}{1 + \sqrt[p]{1-\lambda}}\bigg)
   = \frac{x + \sqrt[p]{\lambda}\, y}{1 + \sqrt[p]{\lambda}},
   \qquad \lambda\in[0,1].
 \end{align*}
Consequently, in the considered special case, the right-hand side of \eqref{help_32} is equal to
 \[
 \sup_{\lambda\in[0,1]}\max
    \bigg\{ \frac{\sqrt[p]{\lambda}\,x + \frac{x+y}{2}}{1 + \sqrt[p]{\lambda}},
             \frac{x + \sqrt[p]{\lambda}\,\frac{x+y}{2} }{1 + \sqrt[p]{\lambda}},
             \frac{x + \sqrt[p]{\lambda}\, y}{1 + \sqrt[p]{\lambda}}  \bigg\}.
 \]
Since
 \[
   \frac{x + \sqrt[p]{\lambda}\,\frac{x+y}{2} }{1 + \sqrt[p]{\lambda}}
           \leq \frac{x + \sqrt[p]{\lambda}\, y}{1 + \sqrt[p]{\lambda}},
           \qquad \lambda\in[0,1],
 \]
 we have that
 \begin{align*}
 \sup_{\lambda\in[0,1]}\max
    \bigg\{ \frac{\sqrt[p]{\lambda}\,x + \frac{x+y}{2}}{1 + \sqrt[p]{\lambda}},
             \frac{x + \sqrt[p]{\lambda}\,\frac{x+y}{2} }{1 + \sqrt[p]{\lambda}},
             \frac{x + \sqrt[p]{\lambda}\, y}{1 + \sqrt[p]{\lambda}}  \bigg\}
   = \sup_{\lambda\in[0,1]}\max
    \bigg\{  \frac{\sqrt[p]{\lambda}\,x + \frac{x+y}{2}}{1 + \sqrt[p]{\lambda}},
             \frac{x + \sqrt[p]{\lambda}\, y}{1 + \sqrt[p]{\lambda}}
    \bigg\}.
 \end{align*}
Note that the inequality
 \[
 \frac{\sqrt[p]{\lambda}\,x + \frac{x+y}{2}}{1 + \sqrt[p]{\lambda}} \leq \frac{x + \sqrt[p]{\lambda}\, y}{1 + \sqrt[p]{\lambda}}
 \]
 is equivalent to $\lambda>\frac{1}{2^p}$.
Consequently,  the right-hand side of \eqref{help_32} is equal to
 \begin{align*}
   \sup_{\lambda\in[0,1]}\max
   & \bigg\{  \frac{\sqrt[p]{\lambda}\,x + \frac{x+y}{2}}{1 + \sqrt[p]{\lambda}},
             \frac{x + \sqrt[p]{\lambda}\, y}{1 + \sqrt[p]{\lambda}}
    \bigg\}
   = \max\bigg\{ \sup_{\lambda\in[0,\frac{1}{2^p}]}  \frac{\sqrt[p]{\lambda}\,x + \frac{x+y}{2}}{1 + \sqrt[p]{\lambda}},
                 \sup_{\lambda\in[\frac{1}{2^p},1]} \frac{x + \sqrt[p]{\lambda}\, y}{1 + \sqrt[p]{\lambda}} \bigg\}\\
  & = \max\bigg\{ \sup_{\mu\in[0,\frac{1}{2}]}  \frac{\mu\,x + \frac{x+y}{2}}{1 + \mu},
                 \sup_{\mu \in[\frac{1}{2},1]} \frac{x + \mu\, y}{1 + \mu} \bigg\}
  = \max\left\{ \frac{x+y}{2}, \frac{x+y}{2} \right\}
    = \frac{x+y}{2},
 \end{align*}
 as desired.
 \proofend
\end{Ex}

\begin{Ex}\label{Ex6}
Let $\alpha,\mu\in\RR_{++}$ and let $\xi$ be a random variable having Lomax (Pareto type II) distribution with parameters $\alpha$ and $\mu$,
 i.e., $\xi$ has a density function
 \[
     f_\xi(x) := \begin{cases}
                  \dfrac{\alpha}{\mu}\left(1 + \dfrac{x}{\mu}\right)^{-(\alpha+1)} & \text{if $x>0$,}\\[2mm]
                  0 & \text{if $x\leq 0$.}
               \end{cases}
 \]
Lomax distribution is frequently used, for example, in actuarial mathematics, and we note that the log-logistic distribution (also known as the Fisk distribution in economics)
is a special case of Lomax distribution, and it is well-applied in survival analysis.
Suppose that $\alpha\in\RR_{++}$ is known.
In Example 6 in Barczy and P\'ales \cite{BarPal3},
 we established the existence and uniqueness of a
 solution of the likelihood equation for $\mu$, and we pointed out the fact that the corresponding function
 \ $\psi:\RR_{++}\times \RR_{++}\to\RR$ takes the form
 \begin{align*}
   \psi(x,\mu) = \frac{\alpha x - \mu}{\mu(\mu + x)}, \qquad x,\mu\in \RR_{++},
 \end{align*}
and it has the properties $[T]$ and $[W]$ (i.e, $[W_n]$ for each $n\in\NN$).
This later fact does not pop up explicitly in Example 6 in Barczy and P\'ales \cite{BarPal3}, but it readily follows, since there we checked that part (vi) of Theorem 1 in Barczy and P\'ales \cite{BarPal2} can be applied.
Further, for each $n\in\NN$ and $x_1,\ldots,x_n\in\RR_{++}$, the likelihood equation for $\mu$ takes the form
 \begin{align*}
   \frac{\alpha}{\alpha+1} = \frac{\mu}{n}\sum_{i=1}^n \frac{1}{x_i+\mu},\qquad \mu\in\RR_{++},
 \end{align*}
 which has a unique solution $\vartheta_{n}(x_1,\ldots,x_n)$.

Next, for $x_1,x_2\in\RR_{++}$ and $(\lambda_1,\lambda_2)\in\Lambda_2$, we solve the equation
 $\lambda_1\psi(x_1,\mu) + \lambda_2\psi(x_2,\mu)=0$, $\mu\in\RR_{++}$, which takes the form
 \[
   \lambda_1\frac{\alpha x_1 - \mu}{\mu(\mu + x_1)} + \lambda_2\frac{\alpha x_2 - \mu}{\mu(\mu + x_2)} = 0,\qquad \mu\in\RR_{++}.
 \]
Multiplying both sides by $\mu(\mu + x_1)(\mu + x_2)$, we have
 \[
   -(\lambda_1+\lambda_2)\mu^2 + \big((\alpha\lambda_1-\lambda_2)x_1 + (\alpha\lambda_2-\lambda_1)x_2 \big)\mu
                               + (\lambda_1+\lambda_2)\alpha x_1x_2=0, \qquad \mu\in\RR_{++}.
 \]
Solving this second-order equation for $\mu$ and taking into account that $\mu\in\RR_{++}$, one can check that
 \begin{align}\label{help_34}
  \begin{split}
  &\vartheta_2^{(\lambda_1,\lambda_2)}(x_1,x_2)
    = \frac{1}{2(\lambda_1+\lambda_2)}
         \Bigg( (\alpha\lambda_1-\lambda_2)x_1 + (\alpha\lambda_2-\lambda_1)x_2 \\
    &\quad  + \!\sqrt{(\alpha\lambda_1-\lambda_2)^2x_1^2 + (\alpha\lambda_2-\lambda_1)^2x_2^2
            + 2\big(\alpha \lambda_1^2 + (\alpha^2 + 4\alpha +1)\lambda_1\lambda_2 + \alpha \lambda_2^2\big)x_1x_2}  \Bigg)
  \end{split}
 \end{align}
 for all $x_1,x_2\in\RR_{++}$ and $(\lambda_1,\lambda_2)\in\Lambda_2$.
Note that $\vartheta_2^{(\lambda_1,\lambda_2)}(cx_1,cx_2) = c \vartheta_2^{(\lambda_1,\lambda_2)}(x_1,x_2)$ for all $c,x_1,x_2\in\RR_{++}$ and $(\lambda_1,\lambda_2)\in\Lambda_2$.
Further, if $\alpha=1$ and $\lambda_1=\lambda_2\in\RR_{++}$,
 then $\vartheta_2^{(\lambda_1,\lambda_2)}(x_1,x_2) = \sqrt{x_1x_2}$ for all $x_1,x_2\in\RR_{++}$, which is nothing else but the geometric mean of $x_1$ and $x_2$.
By \eqref{help_34}, we have an explicit formula for
 \[
  \vartheta_{\ell_1+\ell_2}(\ell_1\odot x_1,\ell_2\odot x_2) = \vartheta_2^{(\ell_1,\ell_2)}(x_1,x_2)
 \]
 for all $x_1,x_2\in\RR_{++}$ and $\ell_1,\ell_2\in\ZZ_+$ with $\ell_1+\ell_2\geq 1$.
Using Theorem \ref{Thm_2_to_n} with $n=3$, for all $x_1,x_2,x_3\in\RR_{++}$, one can approximate
 the unique solution $\vartheta_{3}(x_1,x_2,x_3)$ of the equation
 \begin{align}\label{help_27}
   \frac{\alpha}{\alpha+1} = \frac{\mu}{3}\sum_{i=1}^3 \frac{1}{x_i+\mu},\qquad \mu\in\RR_{++},
 \end{align}
 completely analogously as in Example \ref{Ex3}.

In the special case $\alpha=1$, we have that
 \begin{align*}
     f_\xi(x) = \begin{cases}
                   \frac{1}{\mu\left(1 + \frac{x}{\mu}\right)^2} & \text{if $x>0$,}\\[2mm]
                  0 & \text{if $x\leq 0$,}
               \end{cases}
 \end{align*}
 and, for all $x_1,x_2,x_3\in\RR_{++}$, the equation \eqref{help_27} takes the form
 \[
  \mu\sum_{i=1}^3 \frac{1}{x_i+\mu} = \frac{3}{2},\qquad \mu\in\RR_{++}.
 \]
By algebraic calculations, this equation is equivalent to the equation
 \begin{align*}
   &3\mu^3 + 2(x_1+x_2+x_3)\mu^2 + (x_1x_2 + x_1x_3 + x_2x_3)\mu \\
   &\qquad = \frac{3}{2}\Big( \mu^3 + (x_1+x_2+x_3)\mu^2 + (x_1x_2 + x_1x_3 + x_2x_3)\mu  + x_1x_2x_3 \Big),
       \qquad \mu\in\RR_{++},
 \end{align*}
 which is further equivalent to
 \[
       3\mu^3 + (x_1+x_2+x_3)\mu^2 - (x_1x_2 + x_1x_3 + x_2x_3)\mu - 3x_1x_2x_3 = 0,\qquad \mu\in\RR_{++}.
 \]
In general, this equation can be solved with respect to $\mu$ using the Cardano formula.

In particular, if $x_1:=x$, $x_2:=\sqrt{xy}$, and $x_3:=y$, where $x,y\in\RR_{++}$ with $x<y$,
 then the equation \eqref{help_27} takes the form
 \[
    \frac{1}{3} = \frac{\mu}{3}\left( \frac{1}{x+\mu} + \frac{1}{\sqrt{xy}+\mu} + \frac{1}{y+\mu} \right) , \qquad \mu\in\RR_{++}.
 \]
This equation has only one positive solution, namely, $\mu=\sqrt{xy}$, yielding that
 \[
  \vartheta_3(x,\sqrt{xy},y)=\sqrt{xy}.
 \]
Indeed, one can easily check that $\sqrt{xy}$ is a solution of the equation in question and, as it was explained earlier,
 the equation \eqref{help_27} has a unique solution on $\RR_{++}$.

In what follows, let us choose $x:=1$ and $y:=4$, yielding that $\sqrt{xy}=2$ and $\vartheta_3(1,2,4)=2$.
Next, we calculate $\vartheta_3(1,2,4)$ using Theorem \ref{Thm_2_to_n} as well.
By \eqref{help_34}, we get that
 \begin{align*}
  &\vartheta_2^{(\lambda_1,\lambda_2)}(x_1,x_2)
   = \vartheta_2^{(\lambda_1,\lambda_2)}(1,2)
   = \frac{1}{2(\lambda_1+\lambda_2)}
      \left( \lambda_2 - \lambda_1 + \sqrt{  9\lambda_1^2 + 14\lambda_1\lambda_2 + 9\lambda_2^2} \right),\\
  &\vartheta_2^{(\lambda_1,\lambda_2)}(x_1,x_3)
   = \vartheta_2^{(\lambda_1,\lambda_2)}(1,4)
   = \frac{1}{2(\lambda_1+\lambda_2)}
      \left( 3\lambda_2 - 3\lambda_1 + \sqrt{  25\lambda_1^2 + 14\lambda_1\lambda_2 + 25\lambda_2^2} \right),\\
  &\vartheta_2^{(\lambda_1,\lambda_2)}(x_2,x_3)
   = \vartheta_2^{(\lambda_1,\lambda_2)}(2,4)
   = 2 \vartheta_2^{(\lambda_1,\lambda_2)}(1,2)
   = \frac{1}{\lambda_1+\lambda_2}
      \left( \lambda_2 - \lambda_1 + \sqrt{  9\lambda_1^2 + 14\lambda_1\lambda_2 + 9\lambda_2^2} \right)
 \end{align*}
 for all $(\lambda_1,\lambda_2)\in\Lambda_2$.
Further, recall that
 \[
  \vartheta_{\ell_1+\ell_2}(\ell_1\odot z_1,\ell_2\odot z_2) = \vartheta^{(\ell_1,\ell_2)}_2(z_1,z_2)
 \]
 for all $z_1,z_2\in\RR_{++}$ and $\ell_1,\ell_2\in\ZZ_+$ with $\ell_1+\ell_2\geq 1$.
Hence for all $u,v\in\ZZ_+$ with $u+v\geq 1$, we get that
 \begin{align*}
   \vartheta_{2u+v}(u\odot x_1,(u+v)\odot x_2)
     & = \vartheta^{(u,u+v)}_2(1,2)
      = \frac{1}{2(2u+v)}
         \left(v + \sqrt{  9u^2 + 14 u (u+v) + 9(u+v)^2} \right)\\
     & = \frac{1}{2\left(1+\frac{u}{u+v}\right)}
         \left( 1- \frac{u}{u+v} + \sqrt{  9 \left(\frac{u}{u+v} \right)^2 + 14\cdot\frac{u}{u+v} + 9} \right),
 \end{align*}
  and, analogously,
 \begin{align*}
  & \vartheta_{u+2v}(v\odot x_1,(u+v)\odot x_3)
      = \vartheta^{(v,u+v)}_2(1,4)\\
     & \qquad = \frac{1}{2\left(2-\frac{u}{u+v}\right)}
         \left( 3\cdot\frac{u}{u+v} + \sqrt{  25\left(1-\frac{u}{u+v} \right)^2 + 14\left(1-\frac{u}{u+v}\right) + 25} \right),\\
   &\vartheta_{2u+v}((u+v)\odot x_1,u\odot x_2)
      = \vartheta^{(u+v,u)}_2(1,2)\\
   &\qquad  = \frac{1}{2\left(1+\frac{u}{u+v}\right)}
         \left( -\left(1-\frac{u}{u+v}\right) + \sqrt{  9  + 14\cdot \frac{u}{u+v} + 9\left(\frac{u}{u+v} \right)^2} \right),\\
   &\vartheta_{u+2v}(v\odot x_2,(u+v)\odot x_3)
      = \vartheta^{(v,u+v)}_2(2,4) =2\vartheta^{(v,u+v)}_2(1,2) \\
   &\qquad = \frac{1}{2-\frac{u}{u+v}}
    \left( \frac{u}{u+v} + \sqrt{  9 \left(1-\frac{u}{u+v} \right)^2 + 14\left(1-\frac{u}{u+v}\right) + 9} \right),\\
   &\vartheta_{2u+v}((u+v)\odot x_1,u\odot x_3) = \vartheta^{(u+v,u)}_2(1,4)\\
   &\qquad  = \frac{1}{2\left(1+\frac{u}{u+v}\right)}
         \left( -3\left(1-\frac{u}{u+v}\right)  + \sqrt{  25  + 14\cdot \frac{u}{u+v} + 25\left(\frac{u}{u+v} \right)^2} \right),\\
   &\vartheta_{u+2v}((u+v)\odot x_2,v\odot x_3)
      = \vartheta^{(u+v,v)}_2(2,4)
       = 2\vartheta^{(u+v,v)}_2(1,2) \\
   &\qquad = \frac{1}{2-\frac{u}{u+v}}
         \left( -\frac{u}{u+v}  + \sqrt{  9  + 14\left(1-\frac{u}{u+v}\right) + 9\left(1-\frac{u}{u+v} \right)^2} \right).
 \end{align*}
Consequently, using Theorem \ref{Thm_2_to_n} with $n=3$, \eqref{help_29} in Example \ref{Ex3}
 (of which the derivation we did not use the special form of $\psi$ given in Example \ref{Ex3})
 and \eqref{help_36}, we obtain that
 \begin{align*}
 \vartheta_3(1,2,4) = \sup_{\lambda\in[0,1]}\max\big\{  A(\lambda), B(\lambda), C(\lambda)\big\}
                     = \max\bigg\{  \sup_{\lambda\in[0,1]}A(\lambda), \sup_{\lambda\in[0,1]} B(\lambda), \sup_{\lambda\in[0,1]} C(\lambda)\bigg\},
 \end{align*}
 where
{\footnotesize \begin{align*}
  &A(\lambda)
   := \min\bigg(\frac{1}{2\left(1+\lambda\right)}
            \left( 1- \lambda + \sqrt{  9 \lambda^2 + 14\lambda + 9} \right),
            \frac{1}{2\left(2-\lambda\right)}
              \left( 3\lambda + \sqrt{  25(1-\lambda)^2 + 14(1-\lambda) + 25} \right) \bigg),\\
  &B(\lambda):=\min\bigg(\frac{1}{2\left(1+\lambda\right)}
            \left( -(1- \lambda) + \sqrt{  9 \lambda^2 + 14\lambda + 9} \right),
            \frac{1}{2-\lambda}
              \left( \lambda + \sqrt{  9(1-\lambda)^2 + 14(1-\lambda) + 9} \right) \bigg),\\
  &C(\lambda):=\min\bigg(\frac{1}{2\left(1+\lambda\right)}
            \left( -3(1- \lambda) + \sqrt{  25 \lambda^2 + 14\lambda + 25} \right),
            \frac{1}{2-\lambda}
              \left( -\lambda + \sqrt{  9(1-\lambda)^2 + 14(1-\lambda) + 9} \right) \bigg)
 \end{align*}}
 for all $\lambda\in[0,1]$.
Using the software Wolfram Mathematica 14.1, one can obtain that
 \begin{align*}
   &\sup_{\lambda\in[0,1]}A(\lambda) = 2,\qquad \argmax_{\lambda\in[0,1]}A(\lambda) = \{0\},
 \end{align*}
 \begin{align*}
   &\sup_{\lambda\in[0,1]}B(\lambda) = \sqrt{2},\qquad \argmax_{\lambda\in[0,1]}B(\lambda) = \{1\},\\
   &\sup_{\lambda\in[0,1]}C(\lambda) = 2,\qquad \argmax_{\lambda\in[0,1]}C(\lambda) = \{1\}.
 \end{align*}
This implies that $\vartheta_3(1,2,4) = \max\{2,\sqrt{2},2\}=2$, as desired.
\proofend
\end{Ex}

\begin{Ex}\label{Ex_Expectile}
Let $\alpha\in(0,1)$, $n\in\NN$, and $x_1,\ldots,x_n\in\RR$.
The empirical $\alpha$-expectile based on $x_1,\ldots,x_n$ is defined as any solution of the minimization problem:
 \[
   \min_{t\in\RR} \sum_{i=1}^n \Big(\alpha \bone_{\{x_i\geq t\}} + (1-\alpha) \bone_{\{x_i<t\}} \Big)(x_i-t)^2,
 \]
 see, e.g., Newey and Powell \cite{NewPow}.
Expectiles are also called smoothed versions of quantiles.
Motivated by this minimization problem, we may consider the function $\psi:\RR\times\RR\to\RR$,
 \begin{equation}\label{Exp-theor}
   \psi(x,t):=\begin{cases}
                \alpha(x-t) & \text{if $x>t$,}\\
                0 & \text{if $x=t$,}\\
                (1-\alpha)(x-t) & \text{if $x<t$.}
               \end{cases}
 \end{equation}
As we have seen in Example 7 in Barczy and P\'ales \cite{BarPal2}, $\psi$ is a $W$-function
 (i.e., it is a $W_n$-function for each $n\in\NN$), and $\vartheta_1(x)=x$, $x\in\RR$.
Taking into account that $\psi$ is continuous in its second variable, we also get that $\psi$ is a $Z_n^\blambda$-function
 for each $n\in\NN$ and $\blambda=(\lambda_1,\ldots,\lambda_n)\in\Lambda_n$, i.e.,
 for each $n\in\NN$, $x_1,\ldots, x_n\in\RR$, and $\blambda\in\Lambda_n$,
 the equation $\sum_{i=1}^n \lambda_i \psi(x_i,t) = 0$, $t\in\RR$, has a unique solution,
 denoted by $\vartheta_n^\blambda(x_1,\ldots,x_n)$.
In particular, $\psi\in\Psi[T,W_2](\RR,\RR)$, and hence one can apply Theorem \ref{Thm_2_to_n} to $\psi$.

In what follows, we consider the special case $n=3$.
Using Theorem \ref{Thm_2_to_n} with $n=3$ and  \eqref{help_29} in Example \ref{Ex3}
 (of which the derivation we did not use the special form of $\psi$ given in Example \ref{Ex3}),
 for all $x_1,x_2,x_3\in\RR$, one can approximate the unique solution $\vartheta_3(x_1,x_2,x_3)$ of the equation $\sum_{i=1}^3 \psi(x_i,t)=0$, $t\in\RR$,
 in terms of the quantities
 \begin{align*}
  &\vartheta_{2u+v}( u\odot x_1,(u+v)\odot x_2),
   \qquad \vartheta_{u+2v}(v\odot x_1,(u+v)\odot x_3),
   \qquad \vartheta_{2u+v}((u+v)\odot x_1,u\odot x_2), \\
  & \vartheta_{u+2v}(v\odot x_2,(u+v)\odot x_3),
   \qquad  \vartheta_{2u+v}((u+v)\odot x_1,u\odot x_3),
   \qquad \vartheta_{u+2v}((u+v)\odot x_2,v\odot x_3)
 \end{align*}
 where $u,v\in\ZZ_+$ satisfying $u+v\geq1$.
Next, we derive explicit formulae for the previous quantities.
For all $x_1,x_2\in\RR$ with $x_1<x_2$ and $(\lambda_1,\lambda_2)\in\Lambda_2$, we have that
 \begin{align*}
   \lambda_1\psi(x_1,t) + \lambda_2\psi(x_2,t)
      = \begin{cases}
             -\alpha (\lambda_1+\lambda_2) t + \alpha(\lambda_1x_1+\lambda_2x_2) & \text{if $t<x_1$,}\\
             \lambda_2\alpha(x_2-t)     & \text{if $t=x_1$,}\\
             -(\lambda_1(1-\alpha)+\lambda_2\alpha) t + \lambda_1(1-\alpha)x_1+\lambda_2\alpha x_2 & \text{if $x_1<t<x_2$,}\\
              \lambda_1(1-\alpha)(x_1-t)     & \text{if $t=x_2$,}\\
              -(1-\alpha)(\lambda_1+\lambda_2) t + (1-\alpha)(\lambda_1x_1+\lambda_2x_2) & \text{if $t>x_2$.}
       \end{cases}
 \end{align*}
This shows that, for all $x_1,x_2\in\RR$ with $x_1<x_2$ and $(\lambda_1,\lambda_2)\in\Lambda_2$, the equation
 $\lambda_1\psi(x_1,t)+\lambda_2\psi(x_2,t)=0$, $t\in\RR$, has the unique solution
 \[
    \vartheta_2^{(\lambda_1,\lambda_2)}(x_1,x_2)
         = \frac{\lambda_1(1-\alpha)x_1 + \lambda_2 \alpha x_2}{\lambda_1(1-\alpha) + \lambda_2 \alpha}
                                        \in[x_1,x_2],
 \]
 which is nothing else but a weighted average of $x_1$ and $x_2$.
This implies that
 \begin{align*}
   \vartheta_{\ell_1+\ell_2}(\ell_1\odot x_1,\ell_2\odot x_2) = \vartheta_2^{(\ell_1,\ell_2)}(x_1,x_2)
      = \frac{\ell_1(1-\alpha)x_1 + \ell_2 \alpha x_2}{\ell_1(1-\alpha) + \ell_2 \alpha}
 \end{align*}
 for all $x_1,x_2\in\RR$ with $x_1<x_2$ and $\ell_1,\ell_2\in\ZZ_+$ with $\ell_1+\ell_2\geq 1$.
Finally, we note that
 \[
   \vartheta_2^{(\lambda_1,\lambda_2)}(x,x)
      =  \vartheta_1^{(\lambda_1+\lambda_2)}(x)
      = \vartheta_1(x)=x
 \]
 for all $x\in\RR$ and $(\lambda_1,\lambda_2)\in\Lambda_2$.
\proofend
\end{Ex}

\section{Comparison of generalized $\psi$-estimators using comparative functions}
\label{section_comparative}

This part is motivated by Section 4 in P\'ales \cite{Pal1984}, where the author investigated the problem of
 comparison of means on a real interval using comparative functions (defined below).

\begin{Def}\label{Def_comparative_function}
Let $\Theta$ be a nondegenerate open interval of $\RR$.
A function $C:\Theta\times \Theta\to\RR$ is called a comparative function on $\Theta$
 if it is continuous and if the functions
 \[
  \Theta\ni s\mapsto C(s,t) \qquad \text{and}\qquad \Theta\ni t\mapsto C(s,t)
 \]
 are increasing and decreasing for any fixed $t\in\Theta$ and $s\in\Theta$, respectively.
\end{Def}

For example, $C:\RR\times\RR\to\RR$, $C(s,t):=s-t$, $s,t\in\RR$, is a comparative function on $\RR$,
 and $C:\RR_{++}\times\RR_{++}\to\RR$, $C(s,t):=\frac{s}{t}$, $s,t\in\RR_{++}$, is a comparative function on $\RR_{++}$.

\begin{Thm}\label{Thm_comparative_function}
Let $C:\Theta\times \Theta\to\RR$ be a comparative function on $\Theta$. Let $\psi\in\Psi[T,W_2](X,\Theta)$ and $\varphi\in\Psi[T](X,\Theta)$.
Then the following two assertions are equivalent:
 \begin{itemize}
  \item[(i)] For all $n\in\NN$ and $x_1,\ldots,x_n\in X$,
  \[
     C\big(\vartheta_{n,\psi}(x_1,\ldots,x_n), \vartheta_{n,\varphi}(x_1,\ldots,x_n) \big)\leq 0.
  \]
  \item[(ii)] For all $k,m\in\ZZ_+$ with $k+m\geq 1$ and $x,y\in X$,
  \[
     C\big(\vartheta_{k+m,\psi}(k\odot x,m\odot y), \vartheta_{k+m,\varphi}(k\odot x,m\odot y) \big)\leq 0.
  \]
 \end{itemize}
\end{Thm}

\begin{proof}
The implication $(i)\Longrightarrow(ii)$ is obvious by taking $n=k+m$ and $(x_1,\ldots,x_n)=(k\odot x,m\odot y)$.

Assume now that the assertion (ii) holds.
On the contrary, assume that (i) is not valid.
Then there exist $n\in\NN$ and $x_1,\dots,x_n\in X$ such that
\[
  C\big(\vartheta_{n,\psi}(x_1,\ldots,x_n), \vartheta_{n,\varphi}(x_1,\ldots,x_n) \big)>0.
\]
In view of the continuity of $C$ and the openness of $\Theta$, there  exists an $\varepsilon>0$ such that $\vartheta_{n,\psi}(x_1,\ldots,x_n)-\varepsilon\in\Theta$ and
\[
  C\big(\vartheta_{n,\psi}(x_1,\ldots,x_n)-\varepsilon, \vartheta_{n,\varphi}(x_1,\ldots,x_n) \big)>0.
\]
Using Theorem~\ref{Thm_2_to_n}, we can find a matrix $\mathsf{m}\in\mathsf{M}_n$ such that
\begin{align}\label{help_25}
   \vartheta_{n,\psi}(x_1,\dots,x_n)-\varepsilon
   <\min_{(i,j)\in\Delta(\mathsf{m})}
    \vartheta_{\mathsf{m}_{i,j}+\mathsf{m}_{j,i},\psi}(\mathsf{m}_{i,j}\odot x_i, \mathsf{m}_{j,i}\odot x_j).
\end{align}
At the same time, by Step 1 of the proof of Theorem~\ref{Thm_2_to_n+w}, we also have that
\begin{align*}
   \vartheta_{n,\varphi}(x_1,\dots,x_n)
   \geq\min_{(i,j)\in\Delta(\mathsf{m})}
    \vartheta_{\mathsf{m}_{i,j}+\mathsf{m}_{j,i},\varphi}(\mathsf{m}_{i,j}\odot x_i, \mathsf{m}_{j,i}\odot x_j).
\end{align*}
Here we can indeed refer to Step 1 of the proof of Theorem~\ref{Thm_2_to_n+w} even if we only assumed that $\varphi\in\Psi[T](X,\Theta)$
 instead of $\varphi\in\Psi[T,W_2](X,\Theta)$, since for deriving the inequality \eqref{help_18},
  we only used Theorem~\ref{Thm_psi_becsles_mean_prop}, which holds under the weaker assumption in question.
Therefore, there exists $(i,j)\in\Delta(\mathsf{m})$ such that
\begin{align*}
   \vartheta_{n,\varphi}(x_1,\dots,x_n)
   \geq\vartheta_{\mathsf{m}_{i,j}+\mathsf{m}_{j,i},\varphi}(\mathsf{m}_{i,j}\odot x_i, \mathsf{m}_{j,i}\odot x_j).
\end{align*}
On the other hand, in view of \eqref{help_25}, we also have that
\begin{align*}
   \vartheta_{n,\psi}(x_1,\dots,x_n)-\varepsilon
   <\vartheta_{\mathsf{m}_{i,j}+\mathsf{m}_{j,i},\psi}(\mathsf{m}_{i,j}\odot x_i, \mathsf{m}_{j,i}\odot x_j),
\end{align*}
Combining the last two inequalities with the monotonicity properties of $C$, it follows that
\Eq{*}{
   0&<C\big(\vartheta_{n,\psi}(x_1,\ldots,x_n)-\varepsilon, \vartheta_{n,\varphi}(x_1,\ldots,x_n) \big)\\
   &\leq C\big(\vartheta_{\mathsf{m}_{i,j}+\mathsf{m}_{j,i},\psi}(\mathsf{m}_{i,j}\odot x_i, \mathsf{m}_{j,i}\odot x_j), \vartheta_{\mathsf{m}_{i,j}+\mathsf{m}_{j,i},\varphi}(\mathsf{m}_{i,j}\odot x_i, \mathsf{m}_{j,i}\odot x_j) \big)
   \leq 0,
}
where the last inequality is a consequence of the assertion (ii).
The contradiction so obtained proves that the assertion (i) has to be valid.
\end{proof}

\begin{Rem}

(i) Note that Theorem \ref{Thm_comparative_function} yields the equivalence between the parts (i) and (ii) of Theorem 1
 in Barczy and P\'ales \cite{BarPal3} provided that $\psi \in\Psi[T,W_2](X,\Theta)$ and $\varphi \in\Psi[T](X,\Theta)$
 (one can choose $C(s,t):=s-t$, $s,t\in\Theta$).
We call the attention to the fact that in Theorem 1 in Barczy and P\'ales \cite{BarPal3},
 we supposed that $\psi \in\Psi[T,Z_1](X,\Theta)$ and $\varphi \in\Psi[Z](X,\Theta)$.

(ii) The proof of Theorem \ref{Thm_comparative_function} shows that the continuity of $C$ was not fully used, we only needed that $C$ is continuous in its first variable.
We also point out the fact that our proof technique is completely different from that of the corresponding result for means on pages 68-70 in P\'ales \cite{Pal1984}.
\proofend
\end{Rem}

\begin{Cor}\label{Cor_comparative_function}
Let $C:\Theta\times \Theta\to\RR$ be a comparative function on a non-degenerate open interval $\Theta$.
Let $\psi\in\Psi[T,W_2](X,\Theta)$ and $\varphi\in\Psi[T](X,\Theta)$.
Then
 \begin{align}\label{help_ineq_comparative}
  \begin{split}
  & \sup_{n\in\NN,\, x_1,\ldots,x_n\in\Theta} C\big(\vartheta_{n,\psi}(x_1,\ldots,x_n), \vartheta_{n,\varphi}(x_1,\ldots,x_n) \big) \\
  & \qquad  = \sup_{x,y\in X, \; k,m\in\ZZ_+\; {with}\; k+m\geq 1}
          C\big(\vartheta_{k+m,\psi}(k\odot x,m\odot y), \vartheta_{k+m,\varphi}(k\odot x,m\odot y) \big).
  \end{split}
 \end{align}
\end{Cor}

\begin{proof}
It is obvious that the left hand side of \eqref{help_ineq_comparative} is greater than or equal to its right hand side. To verify the reversed inequality,
let $c^*$ denote the right hand side of \eqref{help_ineq_comparative}.
If $c^*=\infty$, then there is nothing to prove. Thus, we may assume that $c^*$ is finite.
Applying Theorem \ref{Thm_comparative_function} (the implication $(ii)\Rightarrow(i)$) with the comparative function $C-c^*$,
 we get that the left hand side of \eqref{help_ineq_comparative} is less than or equal to $c^*$.
This establishes the equality in \eqref{help_ineq_comparative}.
\end{proof}

In the next example, we give an application of Corollary \ref{Cor_comparative_function}, we derive Schweitzer's inequality (also called Kantorovich's inequality),
 see, e.g., Bullen et al.\ \cite[formula (14), page 202]{Bul03}.

\begin{Ex}
Let $\Theta:=\RR_{++}$, $X:=[a,b]$, where $0<a<b$, and $C:\RR_{++}\times\RR_{++}\to\RR$, $C(s,t):=\frac{s}{t}$, $s,t\in\RR_{++}$.
Then $C$ is a comparative function on $\RR_{++}$.
Let $\psi:[a,b]\times\RR_{++}\to\RR$, $\psi(x,t):=x-t$, $x\in[a,b]$, $t\in\RR_{++}$, and $\varphi:[a,b]\times\RR_{++}\to\RR$,
 $\varphi(x,t):=\frac{1}{t} - \frac{1}{x}$, $x\in[a,b]$, $t\in\RR_{++}$.
Then $\psi,\varphi\in\Psi[T,W_2]([a,b],\RR_{++})$, and, for each $n\in\NN$ and $x_1,\ldots,x_n\in[a,b]$, we have that
 \[
   \vartheta_{n,\psi}(x_1,\ldots,x_n)=\frac{x_1+\cdots +x_n}{n} \qquad \text{and}
      \qquad  \vartheta_{n,\varphi}(x_1,\ldots,x_n)=\frac{n}{\frac{1}{x_1}+\cdots + \frac{1}{x_n}},
 \]
i.e., $\vartheta_{n,\psi}(x_1,\ldots,x_n)$ and $\vartheta_{n,\varphi}(x_1,\ldots,x_n)$ are the arithmetic and harmonic means of $x_1,\ldots,x_n$, respectively.
Further, for each $x_1,x_2\in[a,b]$ and $\blambda=(\lambda_1,\lambda_2)\in\Lambda_2$, we have that
 \[
   \vartheta_{2,\psi}^\blambda(x_1,x_2)=\frac{\lambda_1}{\lambda_1+\lambda_2}x_1 + \frac{\lambda_2}{\lambda_1+\lambda_2}x_2 \qquad \text{and}
      \qquad  \vartheta_{n,\varphi}^\blambda(x_1,x_2)=\frac{1}{ \frac{\lambda_1}{\lambda_1+\lambda_2} \cdot \frac{1}{x_1} + \frac{\lambda_2}{\lambda_1+\lambda_2} \cdot  \frac{1}{x_2}}.
 \]
Next, we derive Schweitzer's inequality (see, e.g., formula (14) on page 202 in Bullen et al.\ \cite{Bul03}), which states that
 \[
  \frac1{n^2}(x_1+\cdots+x_n)\left(\frac{1}{x_1}+\cdots + \frac{1}{x_n} \right)
         \leq \frac{(a+b)^2}{4ab}, \qquad
         n\in\NN,\quad \,x_1,\dots,x_n\in[a,b].
 \]
By Corollary \ref{Cor_comparative_function}, we have
 \begin{align*}
    & \sup_{n\in\NN,\,x_1,\dots,x_n\in[a,b]}  \,\,\frac1{n^2}(x_1+\cdots+x_n)\left(\frac{1}{x_1}+\cdots + \frac{1}{x_n} \right)
     = \sup_{n\in\NN,\,x_1,\dots,x_n\in[a,b]} \frac{\vartheta_{n,\psi}(x_1,\ldots,x_n)}{\vartheta_{n,\varphi}(x_1,\ldots,x_n)} \\
    &\qquad  = \sup_{x,y\in[a,b],\; k,m\in\ZZ_+ k,m\in\ZZ_+ \; {with}\; k+m\geq 1 }
            \frac{\vartheta_{2,\psi}(k\odot x, m\odot y)}{\vartheta_{2,\varphi}(k\odot x, m\odot y)}\\
    &\qquad  = \sup_{x,y\in[a,b],\; k,m\in\ZZ_+ k,m\in\ZZ_+ \; {with}\; k+m\geq 1 }
             \left( \frac{k}{k+m} x + \frac{m}{k+m} y \right) \left(  \frac{k}{k+m} \cdot \frac{1}{x} + \frac{m}{k+m}\cdot \frac{1}{y}\right)\\
    &\qquad = \sup_{x,y\in[a,b],\; t\in[0,1]} (tx+(1-t)y)\left(t\cdot\frac{1}{x} + (1-t)\cdot\frac{1}{y}\right)
 \end{align*}
  \begin{align*}
    &\qquad = \sup_{x,y\in[a,b],\; t\in[0,1]} \left( t^2 + t(1-t)\left( \frac{x}{y} + \frac{y}{x} \right) + (1-t)^2\right)\\
    &\qquad = \sup_{t\in[0,1]} \left( t^2 + t(1-t) \sup_{x,y\in[a,b]}\left( \frac{x}{y} + \frac{y}{x} \right) + (1-t)^2\right).
 \end{align*}
Since the function $[\frac{a}{b},\frac{b}{a}]\ni u\mapsto u+\frac{1}{u}$ is monotone decreasing on $[\frac{a}{b},1]$ and monotone increasing
 on $[1,\frac{b}{a}]$, we get that
  \[
    \sup_{x,y\in[a,b]} \left(\frac{x}{y} + \frac{y}{x} \right)
        = \sup_{u\in[\frac{a}{b},\frac{b}{a}]} \left(u + \frac{1}{u} \right)
        = \frac{a}{b}+\frac{b}{a}.
  \]
Hence
 \[
   \sup_{n\in\NN,\,x_1,\dots,x_n\in[a,b]}  \,\,\frac1{n^2}(x_1+\cdots+x_n)\left(\frac{1}{x_1}+\cdots + \frac{1}{x_n} \right)
     = \sup_{t\in[0,1]} \left( t^2 + t(1-t)\left( \frac{a}{b} + \frac{b}{a} \right) + (1-t)^2\right).
 \]
Let $f:[0,1]\to\RR$, $f(t):= t^2 + t(1-t)\left( \frac{a}{b} + \frac{b}{a} \right) + (1-t)^2$, $t\in[0,1]$.
Then
 \[
   f'(t) = 2t+(1-2t)\left( \frac{a}{b} + \frac{b}{a} \right)-2(1-t)=0, \qquad t\in(0,1)
 \]
 holds if and only if $(2t-1)\left( \frac{a}{b} + \frac{b}{a} - 2\right)=0$, where $\frac{a}{b} + \frac{b}{a} - 2>0$,
 and hence we have that the previous equality can hold only with $t=\frac{1}{2}$.
Since $f''(t) = 4-2\big(\frac{a}{b} + \frac{b}{a} \big)<0$, $t\in(0,1)$, $f(0)= f(1) =1$, and
 \[
   f\left(\frac{1}{2}\right)
    = \frac{1}{4} + \frac{1}{4}\left( \frac{a}{b} + \frac{b}{a} \right) + \frac{1}{4}
     =  \frac{(a+b)^2}{4ab}>1,
 \]
 we obtain that
 \begin{align*}
  \sup_{n\in\NN,\,x_1,\dots,x_n\in[a,b]}  \,\,\frac1{n^2}(x_1+\cdots+x_n)\left(\frac{1}{x_1}+\cdots + \frac{1}{x_n} \right)
  = f\left(\frac{1}{2}\right) = \frac{(a+b)^2}{4ab},
 \end{align*}
as desired. This equality also shows that the constant on the right hand side of Schweitzer's inequality is the smallest possible.
\proofend
\end{Ex}

\section*{Acknowledgements}

We acknowledge the valuable suggestions from the referees.

\section*{Declaration of competing interest}

The authors declare that they have no known competing financial interests or personal relationships
that could have appeared to influence the work reported in this paper.

\bibliographystyle{plain}

\end{document}